\documentclass[oneside]{amsart}

\usepackage{amsfonts, amsmath, amssymb, amsthm,mathtools, stmaryrd}
\usepackage{verbatim}
\usepackage[normalem]{ulem}
\usepackage{tikz-cd}
\usepackage{enumitem}

\usepackage{array}
\usepackage{hyperref}

\theoremstyle{plain}
\newtheorem{thm}{Theorem}[section]
\newtheorem{cor}[thm]{Corollary}
\newtheorem{lemma}[thm]{Lemma}

\newtheorem{prop}[thm]{Proposition}

\numberwithin{equation}{section}

\theoremstyle{definition}
\newtheorem{defn}[thm]{Definition}
\newtheorem{example}[thm]{Example}
\newtheorem{rmk}[thm]{Remark}
\newtheorem{question}[thm]{Question}

\theoremstyle{remark}

\newcommand{\BC}{{\mathbb{C}}}

\newcommand{\BG}{{\mathbb{G}}}

\newcommand{\BQ}{{\mathbb{Q}}}
\newcommand{\BR}{{\mathbb{R}}}

\newcommand{\BZ}{{\mathbb{Z}}}

\newcommand{\CA}{{\mathcal A}}
\newcommand{\CB}{{\mathcal B}}
\newcommand{\CC}{{\mathcal C}}
\newcommand{\CD}{{\mathcal D}}
\newcommand{\CE}{{\mathcal E}}
\newcommand{\CF}{{\mathcal F}}

\newcommand{\CH}{{\mathcal H}}

\newcommand{\CL}{{\mathcal L}}
\newcommand{\CM}{{\mathcal M}}
\newcommand{\CN}{{\mathcal N}}
\newcommand{\CO}{{\mathcal O}}
\newcommand{\CP}{{\mathcal P}}

\newcommand{\CT}{{\mathcal T}}

\newcommand{\CW}{{\mathcal W}}

\newcommand{\CY}{{\mathcal Y}}

\newcommand{\Fm}{{\mathfrak{m}}}

\newcommand{\pt}{{\mathsf{p}}}
\newcommand{\ch}{{\mathrm{ch}}}
\newcommand{\td}{{\mathrm{td}}}

\DeclareFontFamily{OT1}{rsfs}{}
\DeclareFontShape{OT1}{rsfs}{n}{it}{<-> rsfs10}{}
\DeclareMathAlphabet{\curly}{OT1}{rsfs}{n}{it}
\renewcommand\hom{\curly H\!om}
\newcommand\ext{\curly Ext}
\newcommand\Ext{\operatorname{Ext}}
\newcommand\Hom{\operatorname{Hom}}
\newcommand\End{\operatorname{End}}
\newcommand{\Aut}{\operatorname{Aut}}
\newcommand{\p}{\mathbb{P}}

\newcommand\Spec{\operatorname{Spec}}

\newcommand{\Coh}{\mathrm{Coh}}

\newcommand{\Pic}{\mathop{\rm Pic}\nolimits}

\newcommand{\ST}{\mathsf{ST}}
\newcommand{\T}{\mathsf{T}}
\newcommand{\Hilb}{\mathsf{Hilb}}

\newcommand{\Sym}{\mathrm{Sym}}
\newcommand{\Span}{\mathrm{Span}}
\newcommand{\Stab}{\mathrm{Stab}}

\newcommand{\id}{\mathrm{id}}

\newcommand{\FM}{\mathsf{FM}}

\newcommand{\F}{\mathsf{F}}
\newcommand{\G}{\mathsf{G}}

\newcommand{\St}{\mathfrak{St}}
\newcommand{\Grpds}{\mathfrak{Grpds}}
\newcommand{\GSt}{{G\textup{-}\mathfrak{St}}}

\newcommand{\CAob}{\CA_{\omega,\beta}}
\newcommand{\Zob}{Z_{\omega,\beta}}

\newcommand{\Cats}{\mathfrak{Cats}}
\newcommand{\GCats}{G\textup{-}\mathfrak{Cats}}

\newcommand{\Ker}{\mathrm{Ker}}
\newcommand{\Ind}{\mathrm{Ind}}

\newcommand{\LGalg}{\Lambda^G_{\textup{alg}}}

\begin{document}
\title[Equivariant categories and fixed loci]{Equivariant categories of symplectic surfaces and fixed loci of Bridgeland moduli spaces}

\author{Thorsten Beckmann}
\address{Universit\"at Bonn, Mathematisches Institut}
\email{beckmann@math.uni-bonn.de}

\author{Georg Oberdieck}
\address{Universit\"at Bonn, Mathematisches Institut}
\email{georgo@math.uni-bonn.de}
\date{\today}

\begin{abstract}\hspace{-3pt}
Given an action of a finite group $G$ on the derived category of a smooth projective variety $X$ we relate the fixed loci of the induced $G$-action on moduli spaces of stable objects in $D^b(\mathrm{Coh}(X))$ with moduli spaces of stable objects in the equivariant category $D^b(\mathrm{Coh}(X))_G$. As an application we obtain a criterion for the equivariant category of a symplectic action on the derived category of a symplectic surface to be equivalent to the derived category of a surface. This generalizes the derived McKay correspondence, and yields a general framework for describing fixed loci of symplectic group actions on moduli spaces of stable objects on symplectic surfaces.
\end{abstract}

\maketitle
\vspace{-15pt}
\setcounter{tocdepth}{1} 
\tableofcontents

\section{Introduction} 
\subsection{Equivariant categories} \label{subsec:main results}
Let $S$ be a smooth complex projective surface which is symplectic, hence either a K3 or abelian surface.
Whenever a finite group $G$ acts symplectically on $S$,
the derived McKay correspondence provides an equivalence between
the category $D^b(S)_G$
of $G$-equivariant objects in the derived category $D^b(S)$,
and the derived category of the minimal resolution of the quotient $S/G$.
The equivariant category $D^b(S)_G$ depends only on the action of $G$ on the derived category and not on the underlying surface. Hence we may ask
whether a similar correspondence can be formulated for group actions on the derived category
which do not come from an action on the surface.
Our first result considers this question under the following assumptions:

Let $\rho$ be the action of a finite group $G$ on $D^b(S)$ satisfying the following conditions:
\begin{enumerate}
\item[(i)] For every $g \in G$ the equivalence $\rho_g \colon D^b(S) \to D^b(S)$ is symplectic.
\item[(ii)] There exists a stability condition $\sigma \in \Stab^{\dagger}(S)$ which is fixed by every $\rho_g$.
\item[(iii)] The group $G$ acts faithfully, \emph{i.e.}\ the equivariant category is indecomposable.
\end{enumerate}
Here an equivalence 
is \emph{symplectic} if the induced action on singular cohomology $H^\ast(S,\BZ)$
preserves the class of the symplectic form.
We let $\Stab^{\dagger}(S)$
be the distinguished connected component of the space of Bridgeland stability conditions of $D^b(S)$ introduced in \cite{BridgelandK3}.
The action $\rho$ is faithful, if $\rho_g \not\cong \id$ for all $g \neq 1$.
Also no generality is lost by assuming (iii)
since for non-faithful actions the equivariant category
decomposes as an orthogonal sum where each summand is determined by a faithful action on $D^b(S)$, see \cite{Notes}.
By the derived Torelli theorem for symplectic surfaces \cite[Thm.\ 0.1]{HuybrechtsConway}, group actions satisfying these conditions
can be constructed using lattice methods.
In particular, there are many such group actions
which do not arise from automorphisms of the surface even after deformation.

Write $\Lambda = H^{2 \ast}(S,\BZ)$ for the even cohomology lattice and let
$\Lambda_{\mathrm{alg}}^G$ be the invariant sublattice of the induced $G$-action on its algebraic part 
\[ \Lambda_{\mathrm{alg}} = \Lambda \cap (H^0(S,\BC) \oplus H^{1,1}(S,\BC) \oplus H^{4}(S,\BC)). \]
Let $M_{\sigma}(v)$ be a moduli space of $\sigma$-semistable objects of Mukai vector $v \in \LGalg$.
For the induced $G$-action on $M_{\sigma}(v)$ we prove the following:

\begin{thm} \label{mainthm1}
Assume that $M_{\sigma}(v)$ is a fine moduli space
and that the fixed locus $M_{\sigma}(v)^G$ has a $2$-dimensional $G$-linearizable connected component $F$.
Then there exists a subgroup $H \subset G^{\vee} = \Hom(G, \BC^{\ast})$, a connected $H$-torsor $S' \to F$ and an equivalence
\[ D^b(S') \xrightarrow{\cong} D^b(S)_G. \]
\end{thm}

We say here that a connected component of $M_{\sigma}(v)^G$ is \emph{$G$-linearizable} if
for some (or equivalently any) point on it the corresponding $G$-invariant object in $D^b(S)$ admits a $G$-linearization.
By a result of Ploog \cite{Ploog} the obstruction to such a linearization is an element in the second group cohomology $H^2(G, \BC^{\ast})$.
Hence for groups where this cohomology vanishes, such as cyclic groups, the condition on $F$ to be $G$-linearizable is automatically satisfied.

Recall from \cite{BM,HuyLehn} that every fine moduli space $M_{\sigma}(v)$ is smooth and inherits a symplectic form from the surface $S$.
By assumption (i) the $G$-action preserves this symplectic form.
Hence, its fixed locus is smooth and symplectic, so $S'$ is a symplectic surface.
%
%
If the action of $G$ is induced by an action on the surface $S$,
then Theorem~\ref{mainthm1} recovers the usual derived McKay correspondence by taking the moduli space to be the Hilbert scheme of points
$\Hilb^{|G|}(S)$
(the component $F$ is the closure of the locus of free orbits).

Theorem~\ref{mainthm1} applies also to coarse moduli spaces $M_{\sigma}(v)$ of stable objects 
with the only difference that $D^b(S')$
has to be replaced by the derived category of $\alpha$-twisted coherent sheaves $D^b(S',\alpha)$,
where $\alpha \in \mathrm{Br}(S')$ is the Brauer class obtain from the universal family of $M_{\sigma}(v)$ by restriction.
For a more general version of the theorem which applies also to moduli spaces containing strictly semistable points, see Section~\ref{subsec:stronger version of thm1}.

\subsection{Fixed loci}
The result above relies on a general relationship between fixed loci of moduli spaces of (semi)stable objects and the equivariant category.

Let $X$ be a smooth projective variety and let 
\[ \Stab^{\ast}(X) \subset \Stab(X) \]
be a connected component of the space of stability conditions satisfying the technical condition ($\dagger$) of Section~\ref{sec:summary}.
The existence of components $\Stab^{\ast}(X)$ satisfying ($\dagger$)
is known for arbitrary curves and surfaces, as well as for certain threefolds, 
see \cite[Rem.\ 26.4]{StabinFamilies} and references therein.
Moreover, 
as shown in \cite{AHLH} there exists good moduli spaces of semistable objects with respect to any stability condition in $\Stab^{\ast}(X)$.

Consider an action on $D^b(X)$ by a finite group $G$. 
Any $G$-invariant stability condition $\sigma \in \Stab(X)$
yields an induced stability condition $\sigma_G$ on the equivariant category \cite{MMS}.
If moreover $\sigma \in \Stab^{\ast}(X)$, then
we will prove that
there exists proper good moduli spaces $M_{\sigma_G}(v')$ of $\sigma_G$-semistable objects
in $D^b(X)_G$, see Theorem~\ref{thm:existence gms}.

\begin{thm} \label{thm:summary}
Let $\sigma \in \Stab^{\ast}(X)$ be $G$-invariant
and let $M$ be a smooth good moduli space of $\sigma$-stable objects
in $D^b(X)$ of class $v \in K(D^b(X))^G$.
Then the natural morphism 
\begin{equation} 
\bigsqcup_{v' \mapsto v} 
M_{\sigma_G}(v') \, \to\, M^G \label{thm:summary:map} \end{equation}
is a $G^{\vee}$-torsor over the union of all $G$-linearizable connected components of $M^G$.
Here $v'$ runs over all classes in $K(D^b(X)_G)$ mapping to $v$ under the forgetful functor.

Furthermore, \eqref{thm:summary:map} is surjective
if $H^2(G,\BC^{\ast}) = 0$ or, more generally, if the $G$-action on $D^b(X)$ factors through the action of a quotient $G \twoheadrightarrow Q$, such that $G$ is a Schur covering group of $Q$.
\end{thm}

The notion of a Schur covering group will be reviewed in Section~\ref{sec:equiv_cat}.

Theorem~\ref{thm:summary} serves as a bridge between the geometry of the fixed locus $M^G$ and the formal properties of the equivariant category.
Information can flow in both ways:
It can be used to describe moduli spaces of stable objects in the equivariant category in terms of the fixed loci, for example showing projectivity.
This generalizes an approach of Nuer towards the moduli space of stable objects on an Enrique surface \cite{Nuer}.
In the case of Theorem~\ref{mainthm1} it is used to determine the equivariant category.
%
In the opposite direction, if one knows that the equivariant category 
is equivalent to the derived category of a variety whose moduli spaces are well-understood
(e.g.\ a curve, $\p^2$ or a symplectic surface\footnote{Strictly
speaking, for symplectic surfaces one also needs to know that the induced stability condition $\sigma_G$ lies in the distinguished component.
This is proven in Section~\ref{sec:induced stability} if the equivalence is induced by a Fourier--Mukai kernel as in Theorem~\ref{mainthm1}.
}),
then the left hand side of \eqref{thm:summary:map} determines the $G$-linearizable part of the fixed locus up to an \'etale cover.

\subsection{Back to symplectic surfaces}
Consider again a $G$-action on the derived category of a symplectic surface $S$ satisfying (i)-(iii).
Assume that we have an equivalence
\[ D^b(S',\alpha) \xrightarrow{\cong} D^b(S)_{G} \]
for a symplectic surface $S'$ with Brauer class $\alpha \in \mathrm{Br}(S')$.
Let $v \in \Lambda_{\mathrm{alg}}^G$ and define
\[ R_v = \{ v' \in \Lambda_{(S',\alpha),\mathrm{alg}} \mid v' \mapsto v \}, \]
where the algebraic part $\Lambda_{(S',\alpha), \mathrm{alg}}$ of the lattice $H^{2\ast}(S',\BZ)$
is taken with respect to $\alpha$ \cite{HuybrechtsStellariHodge}.
If $M_{\sigma}(v)$ is a moduli space of stable objects,
then Theorem~\ref{thm:summary} shows that
\begin{equation*} \bigsqcup_{v' \in R_v} M_{\sigma_G}(v') \to M_{\sigma}(v)^G \end{equation*}
is a $G^{\vee}$-torsor over the union of all $G$-linearizable components.

In a special case we can be more precise. Consider a set of representatives
\[ \overline{R}_v \subset \Lambda_{(S',\alpha), \mathrm{alg}} \]
for the coset $R_v / G^{\vee}$
where the $G^{\vee}$-action is induced by the action on the equivariant category by twisting the linearization, see Section~\ref{sec:equiv_cat}.

\begin{thm} \label{mainthm4}
Suppose that $G$ is cyclic and that $S'$ is a K3 surface.
If $M_{\sigma}(v)$ is a moduli space of stable objects,
then we have an isomorphism
\begin{equation} M_{\sigma}(v)^G \cong \bigsqcup_{v' \in \overline{R}_v} M_{\sigma_G}(v'). \label{esdf} \end{equation}
\end{thm}

Our description of fixed loci can be applied whenever a group action on a moduli space of stable objects
is induced by a group action on the derived category.
Fortunately, it is an immediate consequence of work of Mongardi \cite{Mongardi2}, Huybrechts \cite{HuybrechtsConway}, and Bayer--Macr\'i \cite{BM}
that for K3 surfaces every symplectic group action is of this type. One has the following:

\begin{prop} \label{mainprop}
Let $S$ be a K3 surface and let $\sigma' \in \Stab^{\dagger}(S)$ be a stability condition.
Let $G$ be a finite group which acts faithfully and symplectically on
a moduli space $M = M_{\sigma'}(v)$ of $\sigma'$-stable objects. Then the following holds:
\begin{enumerate}
\item[(a)] There exists a surjection $G' \to G$ from a finite group $G'$
and an action of $G'$ on $D^b(S)$ satisfying (i), (ii) of Section~\ref{subsec:main results}
which induces the given $G$-action on $M$.
\item[(b)] If $G$ is cyclic, then we can take $G' = G$ in part (a).
\end{enumerate}
\end{prop}

The results presented above yield a general framework to determine the fixed loci of any symplectic group action on a moduli space $M$ of stable objects on a symplectic surface $S$.
There are three steps that have to be taken:
\begin{enumerate}
\item[Step 1.] Find the group action on the derived category which induces the action on $M$ (Proposition~\ref{mainprop}).
\item[Step 2.] Determine the equivariant category\footnote{
In the non-cyclic case with respect to a Schur cover of the group}, i.e.\ express it
in terms of derived categories of symplectic surfaces
(Theorem~\ref{mainthm1}).
\item[Step 3.] Apply Theorem~\ref{thm:summary}.
\end{enumerate}
In other words, we have reduced the problem of describing fixed loci of such symplectic actions to determining the equivariant category.
An example where the above process is applied in a non-trivial case can be found in Section~\ref{subsec:order 3 equivalence} below.

\subsection{Related work}
Kamenova, Mongardi, and Oblomkov determined in \cite{KMO} the fixed loci of symplectic involutions
of holomorphic symplectic varieties of K3$^{[n]}$-type. Their argument proceeds by deforming to an involution of the Hilbert scheme of points of a K3 surface
which is induced by an involution on the surface.
For these actions a description of the fixed locus can be obtained by a local analysis near the fixed points. 
Our work here grew out of the desire to also describe fixed loci of more general (e.g.\ non-natural) automorphisms.

By work of Huybrechts \cite{HuybrechtsConway} and
Gaberdiel, Hohenegger, and Volpato \cite{GHV}
there is a bijection between finite groups of symplectic auto-equivalences of a K3 surface fixing a stability condition
and subgroups of the Conway group with invariant lattice of rank at least four.
The bijection generalizes classical work of Mukai \cite{Mukai} relating symplectic automorphism groups of a K3 surface with subgroups of the Mathieu group.
Similar results for abelian surfaces have been obtained by Volpato \cite{Vol}.
In particular, the derived Torelli theorem in \cite[Prop.\ 1.4]{HuybrechtsConway} provides a large reservoir of symplectic group actions on the derived category,
and thus a good testing ground for our ideas. We refer to Section~\ref{sec:examples} for a series of examples.
The auto-equivalences obtained in this way are described lattice-theoretically, but a concrete geometric description is often missing.
By a criterion of Huybrechts \cite{HuybrechtsConway} and Mongardi \cite{Mongardi2} some of these auto-equivalences induce 
an action on a moduli space of stable objects, but not all of them do (it is still an open question whether that criterion is sharp).

Group actions on the derived category also play an important role in the string theory of K3 surfaces.
In physics the pair $(S,\sigma)$ of a symplectic surface and a distinguished stability condition corresponds to
a \emph{non-singular sigma model} on $S$. Symplectic $\sigma$-preserving actions on the derived category 
correspond to supersymmetry-preserving discrete symmetries.
The equivariant categories are the orbifold sigma models.
Based partially on counting BPS states/dyons,
string theory predicts that the orbifold models should be again either K3 or torus (i.e.\ abelian surface) models \cite{PV}.
The relationship between auto-equivalences and the Conway group cited above
provides the key link between BPS counting in equivariant sigma models and moonshine phenomena for the Conway group, see \cite{NoMoreWalls} and \cite{GHV}
for an introduction on the physical and mathematical side respectively.

\subsection{Open questions}
The equivariant categories $D^b(S)_G$ we have considered above are $2$-Calabi--Yau categories.
Moduli spaces of stable objects in them are holomorphic-symplectic varieties of yet unknown type,
and hence 
provide potentially new examples of (irreducible) holomorphic symplectic varieties.
The most pressing question is therefore the following:

\begin{question} \label{main open question}
Is the set of derived categories of (twisted) coherent sheaves on K3 and abelian surfaces closed under the operation of taking equivariant categories with respect to finite group actions satisfying (i)-(iii)?
\end{question}

In this set we should also include deformations of these categories in the sense of \cite{StabinFamilies} such as the Kuznetsov category of a cubic fourfold.
All evidence so far (as well as the expectation of physics) points to a positive answer.
The parallel question in dimension $1$ has an affirmative answer, see \cite[Sec.\ 7]{Notes}. 

\subsection{Plan of the paper}
The paper consists of two parts.
The first part can be read independently and deals with the construction of moduli spaces of objects in the equivariant category.
Section~\ref{sec:equivariant cat} recalls basic properties of equivariant categories.
In Section~\ref{sec:moduli_spaces} we consider the relation between fixed stacks and the equivariant category and prove Theorem~\ref{thm:summary}. 

For the proof we first use Orlov's result on Fourier--Mukai functors \cite{OrlovsThm} to construct a $G$-action
on Lieblich's stack $\mathfrak{M}$ of universally gluable objects in $D^b(X)$ (Section~\ref{subsec:moduli_of_equ}). 
The associated fixed stack $\mathfrak{M}^G$ defined in the categorical sense of Romagny
is precisely the stack of objects in the equivariant category $D^b(X)_G$ (Proposition~\ref{prop_fixed_stack}).
By transferring geometric properties from $\mathfrak{M}$ to its fixed stack
this yields a well-behaved moduli theory for objects in the equivariant category (Section \ref{subsec:ArtinZhang}).
Theorem~\ref{thm:summary} follows then simply by comparing the fixed stack of a $\BG_m$-gerbe with the fixed locus of the underlying coarse moduli space.

The second part concerns equivariant categories of symplectic surfaces.
In Section~\ref{sec:more} we first discuss Serre functors of equivariant categories
and define equivariant Fourier--Mukai transforms.
In Section~\ref{sec:symplectic surfaces} we prove Theorem~\ref{mainthm1} (including its more general form)
and Theorem~\ref{mainthm4}.
In Section~\ref{sec:symplectic surfaces2} we show that in good cases the induced stability condition
lies again in the distinguished component and prove Proposition~\ref{mainprop}.
In Section~\ref{sec:examples} we discuss a series of examples illustrating the general theory.

In Appendix~\ref{sec:app_hearts} we prove that for every distinguished stability condition on a K3 surface
after a 
shift the heart generates the derived category.
In Appendix~\ref{subsec:DTtheory}
we prove a formula for the topological Euler characteristic of the fixed locus of moduli spaces of stable objects on K3 surfaces under cyclic groups actions.

\subsection{Conventions}
We always work over $\BC$. A variety is connected unless specified otherwise.
All functors are derived unless mentioned otherwise.
The $K$-group $K(\CD)$ of a triangulated category $\CD$ with finite-dimensional Hom-spaces
is always taken numerically, i.e.\ modulo the ideal generated by the kernel of the Euler pairing.
Given a smooth projective variety $X$ we let $D^b(X) = D^b (\Coh(X))$ denote the bounded derived category of coherent sheaves on $X$.
If $\pi \colon X \to T$ is a smooth projective morphism with geometrically connected fibers to a $\BC$-scheme $T$,
then $D(X)$ or $D(X/T)$ will stand for the full triangulated subcategory of $T$-perfect complexes of the unbounded derived category of $\CO_X$-modules.
We refer to Sections 2 and 8.1 of \cite{StabinFamilies} for definitions and further references.
If $T = \Spec (\BC)$, then $D(X)$ is the bounded derived category of coherent sheaves as before.

\subsection{Acknowledgements}
We thank Daniel Huybrechts for many discussions on derived categories and K3 surfaces and suggestions on a preliminary version. We also thank Jochen Heinloth for useful comments.
We drew a lot of inspiration for the paper from the string theory of K3 non-linear sigma models.
We thank Albrecht Klemm, Roberto Volpato, and Max Zimet for fruitful discussions and patiently answering our questions.

T.B. was funded by the IMPRS program of the Max-Planck Gesellschaft. G.O. was funded by the Deutsche Forschungsgemeinschaft (DFG) -- OB 512/1-1.

{\large{\part{{Moduli spaces for the equivariant category}}}}

\section{Equivariant categories}
\label{sec:equivariant cat}

\subsection{Categorical actions}
\label{sec:equiv_cat}
An action $(\rho, \theta)$ of a finite group $G$ on an additive $\BC$-linear category $\CD$ consists of
\begin{itemize}
\item for every $g \in G$ an auto-equivalence $\rho_g \colon \CD \to \CD$,
\item for every pair $g,h \in G$ an isomorphism of functors $\theta_{g,h} \colon \rho_{g} \circ \rho_h \to \rho_{gh}$
\end{itemize}
such that for all $g,h,k \in G$ the following diagram commutes
\begin{equation} \label{associativity}
\begin{tikzcd}
\rho_{g} \rho_{h} \rho_{k} \ar{r}{\rho_g  \theta_{h,k}} \ar{d}{\theta_{g,h} \rho_k} & \rho_{g} \rho_{hk} \ar{d}{\theta_{g,hk}} \\
\rho_{gh} \rho_{k} \ar{r}{\theta_{gh,k}} & \rho_{ghk}.
\end{tikzcd}
\end{equation}

A $G$-functor $(f,\sigma) \colon (\CD, \rho, \theta) \to (\CD', \rho', \theta')$ between categories with $G$-actions
is a pair of a functor $f \colon \CD \to \CD'$ together with 2-isomorphisms
$\sigma_g \colon f \circ \rho_g \to \rho_g' \circ f$
such that $(f,\sigma)$ intertwines the associativity relations on both sides, i.e.\ such that the following diagram commutes:
\[
\begin{tikzcd}
f \rho_g \rho_h \ar{d}{f \theta_{g,h}} \ar{r}{\sigma_g \rho_h} & \rho_{g}' f \rho_h \ar{r}{\rho_g' \sigma_h} & \rho_g' \rho_h' f \ar{d}{\theta_{g,h}' f} \\
f \rho_{gh} \ar{rr}{\sigma_{gh}} & & \rho_{gh}' f.
\end{tikzcd}
\]
A $2$-morphism of $G$-functors $(f,\sigma) \to (\tilde{f}, \tilde{\sigma})$ is a $2$-morphism $t\colon f \to f'$ that intertwines the $\sigma_g$,
i.e.\ $\tilde{\sigma}_g \circ t \rho_g = \rho_g' t \circ \sigma_g$.

\begin{defn} \label{defn_equiv_cat}
Given a $G$-action $(\rho, \theta)$ on the category $\CD$
the equivariant category $\CD_{G}$ is defined as follows:
\begin{itemize}
\item Objects of $\CD_G$ are pairs $(E,\phi)$ where $E$ is an object in $\CD$ and $\phi = (\phi_g \colon E \to \rho_{g} E)_{g \in G}$ is a family of isomorphisms
such that
\begin{equation} 
\label{compatibility}
\begin{tikzcd}
E \ar[bend right]{rrr}{\phi_{gh}} \ar{r}{\phi_g} & \rho_{g}E \ar{r}{\rho_g \phi_h} &  \rho_{g} \rho_{h} E \ar{r}{\theta_{g,h}^E} & \rho_{gh} E
\end{tikzcd}
\end{equation} 
commutes for all $g,h \in G$.
\item A morphism from $(E,\phi)$ to $(E', \phi')$ is a morphism $f \colon E \to E'$ in $\CD$
which commutes with linearizations, i.e.\ such that
\[
\begin{tikzcd}
E \ar{r}{f} \ar{d}{\phi_g} & E' \ar{d}{\phi'_g} \\
g E \ar{r}{\rho_gf} & g E'
\end{tikzcd}
\]
commutes for every $g \in G$.
\end{itemize}
\end{defn}

For all objects $(E,\phi)$ and $(E',\phi')$ in $\CD_G$
the group $G$ acts on
$\Hom_{\CD}(E,E')$ via 
$f \mapsto (\phi'_g)^{-1} \circ \rho_g(f) \circ \phi_g$.
By definition,
\[ \Hom_{\CD_G}( (E,\phi), (E,\phi') ) = \Hom_{\CD}(E,E')^G. \]

The equivariant category comes equipped with a forgetful functor
\[ p \colon \CD_G \to \CD, \quad (E,\psi) \mapsto E \]
and a linearization functor
\begin{equation} q\colon \CD \to \CD_G, \quad E \mapsto \left( \oplus_{g \in G} \rho_g E, \phi \right) \label{linearlization_functor} \end{equation}
where the linearization $\phi$ is given by considering
$\theta_{h,h^{-1}g}^{-1} \colon \rho_g E \to \rho_{h} \rho_{h^{-1} g} E$
and then taking the direct sum over all $g$,
\begin{equation} \phi_h = \oplus_g \theta_{h,h^{-1}g}^{-1} \colon \oplus_{g} \rho_g E \to \rho_h \left( \oplus_{g} \rho_{h^{-1} g} E \right) = \rho_h \left( \oplus_{g} \rho_g E \right). \label{can_lin} \end{equation}
By \cite[Lem.\ 3.8]{Elagin}, $p$ is both left and right adjoint to $q$.

We discuss several properties of equivariant categories.
We will often write $g$ for $\rho_g$.

\begin{example}
The \emph{trivial} $G$-action on $\CD$ is defined by $\rho_g = \id$ and $\theta_{g,h}=\id$ for all $g,h \in G$.
In this case the objects of $\CD_G$ are pairs of an object $x \in \CD$ and a homomorphism $\phi \colon G \to \Aut(x)$.
\end{example}

\begin{rmk}
\label{rmk:GroupAction2Rep}
Consider the $2$-category $\GCats$ whose objects are categories with a $G$-action and whose morphisms are $G$-functors.
The equivariant category $\CD_G$
satisfies the universal property
that for all categories $\CA$ we have the equivalence
\[
\Hom_{\Cats}{(\CA, \CD_G)} \cong \Hom_{\GCats}{(\iota(\CA),\CD)}
\]
where $\iota(\CA)$ is the category $\CA$ endowed with the trivial $G$-action. 
Hence, any $G$-functor from $\iota(\CA)$ to $\CD$ factors over the forgetful functor $p$, see \cite[Prop.\ 4.4]{GK} for more details.
\end{rmk}

If a triangulated category has a dg-enhancement,
then the equivariant category is again triangulated \cite[Cor.\ 6.10]{Elagin}.
This is implied also more directly as follows.

\begin{prop} \label{prop_dgenhancement}
Let $\CD$ be a triangulated category with an action of a group $G$. 
Suppose there is a full abelian subcategory $\CA \subset \CD$ such that $D^b(\CA) = \CD$ and $G$ preserves $\CA$, i.e.\ $\rho_g E \in \CA$ for all $E \in \CA$.
Then the following holds.
\begin{enumerate}
\item[(i)] There exist a dg-enhancement $\CD_{dg}$ of $\CD$ together with an action of $G$ on $\CD_{dg}$ which lifts the action of $G$ on $\CD$.
\item[(ii)] The equivariant category $\CD_G$ is triangulated.
\end{enumerate}
\end{prop}
\begin{proof}
By \cite[Sec.\ 1.2]{CS} the dg-quotient category
\[ D_{dg}(\CA) = C_{dg}(\CA) / \mathrm{Acyclic}_{dg}(\CA) \]
of the dg-category of bounded complexes in $\CA$ by the dg-category of acyclic bounded complexes in $A$
defines a dg-enhancement of $D^b(\CA)$.
By hypothesis $D^b(\CA) \cong \CD$ hence $D_{dg}(\CA)$ is a dg-enhancement.
Moreover, 
the $G$-action on $\CD$ induces a $G$-action on $\CA$.
Since $G$ preserves acyclic complexes we obtain a $G$-action on $D_{dg}(\CA)$ with the desired properties.
This proves the first part. 
For the second part we apply \cite{Chen}, see also \cite[Thm.\ 7.1]{Elagin}, to get
\[ \CD_G = D^b(\CA)_G \cong D^b(\CA_G) \]
and as a derived category the latter is naturally triangulated.
\end{proof}

\begin{rmk}
If $X$ is a smooth projective variety, then $D^b(X)$ 
has (up to equivalence) a unique dg-enhancement \cite{LO}.
\end{rmk}

The group of characters
$G^{\vee} = \{ \chi \colon G \to \BC^{\ast} \mid \chi \text{ homomorphism} \}$
acts on the equivariant category $\CD_G$ by the identity on morphisms and by
\[ \chi \cdot (E,\phi) = (E, \chi \phi) \]
on objects, where we let $\chi \phi$ denote the linearization $\chi(g) \phi_g \colon E \to \rho_{g} E$.

An object $E \in \CD$ is called \emph{$G$-invariant} if for all $g \in G$
there exists an isomorphism $\rho_g E \cong E$.
A $G$-linearization of $E$ is an element $\tilde{E} \in \CD_G$ such that $p \tilde{E} \cong E$.
There is the following obstruction for a $G$-invariant simple object to be $G$-linearizable
(which, since $H^2(\BZ_n, \BC^{\ast})=0$ for all $n$, is trivial for cyclic groups).
\begin{lemma}[{\cite[Lem.\ 1]{Ploog}}] \label{lem:Ploog} 
Given a $G$-invariant simple object $E \in \CD$, there exists a class in $H^2(G, \BC^{\ast})$
which vanishes if and only if there exists a $G$-linearization of $E$.
The set of (isomorphism classes) of $G$-linearizations of $E$ is a
torsor under $G^{\vee}$.
\end{lemma}

Example~\ref{example:no G-lin} below shows that this obstruction is effective.

Recall that an extension of groups $1 \to K \to E \to G \to 1$ is \emph{stem} if $K$ is contained both in the commutator subgroup and the center of $E$. Any maximal stem extension $\tilde{G} \twoheadrightarrow G$ is called a Schur covering group of $G$. It has the property that the restriction morphism
\[
H^2(G,\BC^\ast) \to H^2(\tilde{G},\BC^\ast)
\]
vanishes. Hence, by Lemma~\ref{lem:Ploog} if we let $\tilde{G}$ act on $\CD$ via the quotient map to $G$,
then every invariant simple object admits a $\tilde{G}$-linearization.

Let $\Aut \CD$ be the group of equivalences of $\CD$. 
Every group action on $\CD$ yields a subgroup of $\Aut \CD$.
For the converse one has the following obstruction
(which because of $H^3(\BZ_n , \BC^{\ast}) = \BZ_n$ is non-trivial even for cyclic groups).

\begin{lemma} \textup{(}\cite[Sec.\ 2.2]{Notes}\textup{)}  \label{lem:obstruction to action}
Assume that $\Hom(\id_{\CD}, \id_{\CD}) = \BC \id$ and let $G \subset \Aut \CD$ be a finite subgroup.
\begin{enumerate}
\item[(a)] There exists a class in $H^3(G, \BC^{\ast})$ which vanishes if and only if there exists an action of $G$ on $\CD$
whose image in $\Aut \CD$ is $G$.
Moreover, the set of isomorphism classes of such actions is a torsor under $H^2(G,\BC^{\ast})$.
\item[(b)] There exits a finite group $G'$ and a surjection $G' \to G$ such that
$G'$ acts on $\CD$ and the induced map $G' \to \Aut \CD$ is the given quotient map to $G$.
\item[(c)] If $G = \BZ_n$, then we can take $\BZ_{n^2} \to \BZ_n$ in (b).
\end{enumerate}
\end{lemma}

\subsection{Stability conditions} \label{subsec:stab_conditions}
A (Bridgeland) stability condition on a triangulated category $\CD$ is a pair $(\CA, Z)$ consisting of
\begin{itemize}
\item the heart $\CA \subset \CD$ of a bounded $t$-structure on $\CD$ and
\item a stability function $Z\colon K(\CA) \to \BC$
\end{itemize}
satisfying 
several conditions, see \cite{BridgelandStab}.
Given an equivalence $\Phi \colon \CD \to \CD'$
of triangulated categories
the image of $\sigma$ under $\Phi$ is defined by
\[ \Phi \sigma = (\Phi \CA, Z \circ \Phi_{\ast}^{-1}) \]
where $\Phi_{\ast} \colon K(\CD) \to K(\CD')$ is the induced map on $K$-groups.
If $\Phi \colon \CD \to \CD$ is an auto-equivalence, we say that 
$\Phi$ preserves (or fixes) $\sigma$ if $\Phi \sigma = \sigma$.

Let $X$ be a smooth projective variety
together with an action of a finite group $G$ on $D^b(X)$ which fixes a stability condition  $\sigma = (\CA,Z)$.
By \cite[Lem.\ 2.16]{MMS} $\sigma$ induces a stability condition on $D^b(X)_G$ defined by
\[ \sigma_G = (\CA_G, Z_G), \quad Z_G \coloneqq Z \circ p_{\ast} \colon K(\CA_G) \to \BC. \]

\begin{lemma} \label{Lemma_stability}
Let $(E,\phi) \in \CA_G$.
Then $(E,\phi)$ is $\sigma_G$-semistable if and only if $E$ is $\sigma$-semistable.
If $E$ is $\sigma$-stable, then $(E,\phi)$ is $\sigma_G$-stable.
\end{lemma}
\begin{proof}
If an element $E \in \CA_G$ is destabilized by $F$, then $p(E)$ is destabilized by $p(F)$.
Conversely, if $p(E)$ is destabilized by $F' \in \CA$, then the image of the adjoint morphism 
$qF' \to E$
destabilizes $E$.
Hence an element in $(E,\phi) \in \CA_G$ is $\sigma_G$-semistable if and only if $E \in \CA$ is $\sigma$-semistable.
A subobject of $(E,\phi)$ is given by a subobject $F \subset E$ such that $\phi$ restricts to a linearization of $F$.
Hence any destabilizing subobject of $(E,\phi)$ yields a destabilizing subobject of $E$. This shows the second claim.
\end{proof}

\begin{defn} \label{defn_G_sigma_generic}
A class $v \in K(\CA)^G$ is \emph{$(G,\sigma)$-generic}
if it is primitive and for every splitting $v = v_0 + v_1$ with $v_i\in K(\CA)^G\setminus \BZ v$ the summands have different slopes.
\end{defn}

\begin{lemma} \label{lemma_key}
Let $(E,\phi) \in \CA_G$ such that $E$ is $\sigma$-semistable and its class $[E]\in K(\CA)^G$ is $(G,\sigma)$-generic.
Then $(E,\phi)$ is $\sigma_G$-stable. In particular,
\[ \Hom_{\CA_G}( (E,\phi), (E,\phi) ) = \BC \id. \]
\end{lemma}
\begin{proof}
As explained above the object $(E,\phi)$ is $\sigma_G$-semistable.
If it is not stable, then there exists a short exact sequence in $\CA_G$
\[ 0 \to (F_1,\phi) \to (E,\phi) \to (F_2, \phi) \to 0 \]
with $F_1, F_2$ of the same phase as $E$. Applying the forgetful functor we obtain
\[ 0 \to F_1 \to E \to F_2 \to 0 \]
in $\CA$ with $F_i$ semistable of the same phase as $E$.
However, the classes $[F_i]$ are $G$-invariant which shows that $[E] = [F_1]+[F_2]$ is not $(G,\sigma)$-generic.
\end{proof}

\subsection{Fourier--Mukai actions} \label{Subsection_Geometric_case}
Let $\pi \colon X \to T$ be a smooth projective morphism to a $\BC$-scheme $T$ with geometrically connected fibers.
Let
\[ p,q \colon X \times_T X \to X \]
be the projections to the factors.
The Fourier--Mukai transform $\FM_{\CE} \colon D(X) \to D(X)$
with kernel $\CE \in D(X \times_T X)$ is defined by
\[ \FM_{\CE}(A) = q_{\ast}( p^{\ast}(A) \otimes \CE). \]
Using a push-pull argument we have isomorphisms
\begin{equation} \FM_{\CE}( A \otimes \pi^{\ast} B ) \cong \FM_{\CE}(A) \otimes \pi^{\ast} B \label{tp1} \end{equation} 
for all $A \in D(X)$ and $B \in D(T)$, functorial in both $A$ and $B$.

\begin{defn}
A \emph{Fourier--Mukai action} of $G$ on $D(X)$ consists of\footnote{We write $\CE \circ \CF$ to indicate the composition of correspondences $\CE,\CF$.}
\begin{itemize}
\item for every $g \in G$ a Fourier--Mukai kernel $\CE_g \in D(X \times_T X)$,
\item for every pair $g,h \in G$ an isomorphism $\theta_{g,h} \colon \CE_g \circ \CE_h \to \CE_{gh}$
\end{itemize}
such that for all $g,h,k$ the diagram \eqref{associativity} commutes with $\rho_g$ replaced by $\CE_g$.
\end{defn}

For smooth projective varieties we have not defined anything new:

\begin{lemma} \textup{(}\cite[Sec.\ 2.3]{Notes}\textup{)}
\label{lem:FM-action}
Let $X$ be smooth projective variety and let $G$ be a finite group.
Then any $G$-action on $D^b(X)$ is induced by a unique Fourier--Mukai action.
\end{lemma}

Given a Fourier--Mukai action on the derived category of $X/T$
our next goal is to define natural operations on the equivariant category.
If $G$ is induced by an action on $X$, this is discussed in \cite[Sec.\ 4]{BKR}.
Since our $G$-action does not have to preserve the tensor product or the structure sheaf, some care is needed in the general case.

\subsubsection{Pushforward and pullback} \label{subsec:eq-pushforward and pullback}
Consider a fiber product diagram
\begin{equation} \label{diag1}
\begin{tikzcd}
X' \ar{r}{\alpha} \ar{d}{\pi'} & X \ar{d}{\pi} \\
T' \ar{r}{\beta} & T.
\end{tikzcd}
\end{equation}
The pullback of the kernels of the $G$-action on $X$,
\[ (\alpha \times \alpha)^{\ast} \CE_g \in D(X' \times_{T'} X'), \]
together with the pullback of the $\theta_{g,h}$ define a Fourier--Mukai $G$-action on $D(X')$.
We say that the morphism $\alpha$ is \emph{$G$-equivariant}.

Given an equivariant object $(F,\phi)$ in $D(X)_{G}$ we define its pullback by
\[ \alpha^{\ast} (F,\phi) = (\alpha^{\ast} F, \phi') \, \in D(X')_G \]
where the $G$-linearization $\phi'_g$ is the composition
\begin{multline*}
\alpha^{\ast} F \xrightarrow{\alpha^{\ast} \phi_g} \alpha^{\ast}(g E) = \alpha^{\ast} q_{\ast}( p^{\ast}(F) \otimes \CE_g) 
\cong q'_{\ast} (\alpha \times \alpha)^{\ast}( p^{\ast}(F) \otimes \CE_g )  \\
\cong q'_{\ast}( p^{\prime \ast}( \alpha^{\ast} F) \otimes (\alpha \times \alpha)^{\ast} \CE_g ) = g \alpha^{\ast}(F)
\end{multline*}
with $p',q' \colon X' \times_{T'} X' \to X'$ the projections.
The pullback $\alpha^{\ast}$ of an equivariant morphism is the pullback of the morphism in $D(X)$
(one checks that the pullback morphism is $G$-invariant).
Taken together this yields a functor
\[ \alpha^{\ast} \colon D(X)_G \to D(X')_G. \]
Similarly if $\beta$ is proper and flat and $(E,\phi) \in D(X')_G$, we define the pushforward functor by
\[ \alpha_{\ast} (E,\phi) \coloneqq (\alpha_{\ast} E, \phi') \]
where the $G$-linearization $\phi'$ is obtained as the composition
\begin{multline*}
\alpha_{\ast} E \xrightarrow{\alpha_{\ast} \phi_g} \alpha_{\ast} g E
= \alpha_{\ast} q_{\ast}( p^{\ast}(E) \otimes (\alpha \times \alpha)^{\ast}(\CE_g)  )  \\
\cong q'_{\ast} (\alpha \times \alpha)_{\ast} ( p^{\ast}(E) \otimes (\alpha \times \alpha)^{\ast}(\CE_g)  ) 
\cong q'_{\ast}( p^{\prime \ast}( \alpha_{\ast} E) \otimes \CE_g ) = g \alpha_{\ast}(E).
\end{multline*}
The pushforward of an equivariant morphism is the pushforward of the underlying morphism.
The pullback functor $\alpha^{\ast}$ is left adjoint to $\alpha_{\ast}$.

\subsubsection{Hom and tensor product} \label{subsubsection_tensor_product}
Given a $T$-perfect object $B \in D(T)$ and an equivariant object $(E,\phi) \in D(X)_G$ 
we define the tensor product by
\[ (E,\phi) \otimes \pi^{\ast} B \coloneqq (\pi^{\ast} B \otimes E, \phi') \]
where the linearization $\phi'$ is the composition
\[
E \otimes \pi^{\ast}(B) \xrightarrow{\phi_g \otimes \id} \FM_{\CE_g}(E) \otimes \pi^{\ast}(B) \overset{\eqref{tp1}}{\cong} \FM_{\CE_g}(E \otimes \pi^{\ast}(B)) = g (E \otimes \pi^{\ast}(B)). \]
More generally, if $D(T)$ is equipped with the trivial $G$-action and $(B,\chi) \in D(T)_G$, we let
\[ (B, \chi) \otimes (E, \phi) \coloneqq (\pi^{\ast} B \otimes E, \chi \phi') \]

Similarly, given two equivariant objects $(E,\phi)$ and $(F,\psi)$ in $D(X)_G$ and
an open subset $U\subset T$
the group $G$ acts on $\Hom_{D(X_U)}(E|_{U},F|_{U})$
by $f \mapsto \phi_g|_U \circ \FM_{\CE_g|_U}(f) \circ \psi_g^{-1}|_{U}$
where we use again that Fourier--Mukai actions induce actions after base change.
Since this action is compatible with restrictions to smaller open subsets
we obtain a $G$-action on 
$\hom_{\pi}(E,F) \coloneqq \pi_{\ast} \hom(E,F)$
and thus a bifunctor
\[ \hom_{\pi} \colon D(X)_G \times D(X)_G \to D(T)_{G}. \]
It satisfies the usual adjunctions with respect to the tensor product.

For any (closed or non-closed) point $t \in T$ let $\iota_t \colon X_t \to X$ be the inclusion of the fiber of $X$ over $t$.
Given $(E,\phi) \in D(X)_G$ we write $(E,\phi)_t$ for the equivariant pullback $\iota_t^{\ast} (E,\phi)$.
\begin{lemma} \label{Lemma_locally_constant_euler_char}
Let $(E,\phi),(F,\psi)$ be objects in $D(X)_G$. Then
\[ t \mapsto \chi( (E,\phi)_t, (F,\psi)_t) \coloneqq \sum_{i} \dim \Ext^i_{D(X_t)_G}\left( (E,\phi)_t, (F,\psi)_t \right) \]
is locally constant in $t$.
\end{lemma}
\begin{proof}
By a push-pull argument we have that
\[ \chi( (E,\phi)_t, (F,\psi)_t) = \chi( k(t), \hom_{\pi}( (E,\phi), (F,\psi) )^G \otimes k(t) ). \]
Since $\hom_{\pi}( (E,\phi), (F,\psi) )$ is perfect, the same holds for its invariant part which implies the claim.
\end{proof}

\section{Moduli spaces} \label{sec:moduli_spaces}
\subsection{Group actions on stacks} \label{Subsection_group_action_on_Stacks}
Following \cite{R} an action of a finite group $G$ on a stack $\CM$ over $\BC$ consists of
\begin{itemize}
\item for every $g \in G$ an automorphism of stacks $\rho_g \colon \CM \to \CM$
\item for every pair $g,h \in G$ an isomorphism of functors $\theta_{g,h} \colon \rho_{g} \rho_{h} \to \rho_{gh}$
\end{itemize}
such that for all $g,h,k \in G$ the diagram \eqref{associativity} commutes.
In other words, if we view $\CM$ as a category fibered in groupoids,
then a $G$-action on $\CM$ is precisely a $G$-action on the category $\CM$ in the sense of Section~\ref{sec:equiv_cat}
with the additional assumption that every $\rho_g$ is a morphism of stacks.
A morphism of stacks with $G$-actions (also called a $G$-equivariant morphism) is a $G$-functor $(f,\sigma)$ such that $f$ is a morphism of stacks.
A $2$-morphism of such morphisms is a $2$-morphism of $G$-functors.

Let $\St$ and $G$-$\St$ denote the 2-categories of stacks and stacks with a $G$-action respectively.
There is a functor 
$\iota \colon \St \to \GSt$
which equips a stack with the trivial $G$-action. 
Let $\Grpds$ be the category of groupoids.

\begin{defn}[{\cite[Def.\ 2.3]{R}}] \label{defn_fixed_stack}
Let $G$ be a finite group acting on a stack $\CM$.
The fixed stack is the functor $\CM^G \colon \St \to \Grpds$
defined by the condition that for all stacks $T$ we have the equivalence
\[ \Hom_{\St}(T, \CM^G) \cong \Hom_{\GSt}(\iota(T), \CM). \]
\end{defn}

Hence there is a $G$-equivariant morphism $\epsilon \colon \iota(\CM^G) \to \CM$
satisfying the following universal property:
For any stack $T$ and for any $G$-equivariant morphism $f \colon \iota(T) \to \CM$ there exists
a unique morphism $\tilde{f} \colon T \to \CM^G$ such that $\epsilon \circ \tilde{f} = f$.

\begin{rmk} \label{rmk_fixed_stack}
As explained in \cite[Proof of Prop.\ 2.5]{R} the objects of $\CM^G$
are pairs $(x, \{ \alpha_g \}_{g \in G})$ of an element $x \in \CM$ and
maps $\alpha_g \colon x \to g.x$ such that $\theta_{g,h}^x \circ g \alpha_h \circ \alpha_g = \alpha_{gh}$ for all $g,h \in G$.
Morphisms are the morphisms in $\CM$ which respect the linearizations.
Hence, viewed as a category, the fixed stack $\CM^G$
\emph{is} the equivariant category $\CM_G$ of the action $(\rho, \theta)$ in the sense of Definition~\ref{defn_equiv_cat}. 
This can be seen also more conceptually: By the universal property of the equivariant category
(Definition~\ref{defn_fixed_stack}) we have a functor $\CM^G \to \CM_G$,
but by the universal property of the fixed stack we also have an inverse.
\end{rmk}

\begin{rmk} \label{rmk:fixed stack base change}
By the universal property, if $(f,\sigma) \colon \CN \to \CM$ is a $G$-equivariant morphism which is a monomorphism
(e.g.\ an open or closed immersion), then we have a fiber diagram
\[
\begin{tikzcd}
\CN^G \ar{r} \ar{d}{\epsilon} & \CM^G \ar{d}{\epsilon} \\
\CN \ar{r}{f} & \CM.
\end{tikzcd}
\]
\end{rmk}

\begin{prop} \label{algebraicity_of_fixed_stack}
{\cite[Thm.\ 3.3, 3.6]{R}}
Let $G$ be a finite group acting on an Artin stack $\CM$ (locally) of finite type over $\BC$.
Then $\CM^G$ is an Artin stack (locally) of finite type over $\BC$ and
the classifying morphism
$\epsilon \colon \CM^G \to \CM$
is representable, separated and quasi-compact.
If $\CM$ has affine diagonal, then so does $\CM^G$.

Furthermore, consider any property of morphisms of schemes that is satisfied by closed immersions and is stable under composition.
Then, if the diagonal of $\CM$ has this property, then $\epsilon$ has this property.
\end{prop}
\begin{proof}
We prove that $\CM^G$ has affine diagonal if $\CM$ has. Everything else can be found in \cite{R}.
Assume that $\CM$ has affine diagonal and consider the commutative diagram
\[
\begin{tikzcd}
\CM^G \ar{r}{\Delta_{\CM^G}} \ar[swap]{dr}{\Delta_{\CM} \circ \epsilon} & \CM^G \times \CM^G \ar{d}{\epsilon \times \epsilon} \\
& \CM \times \CM.
\end{tikzcd}
\]
Since $\Delta_{\CM}$ is affine, $\epsilon$ is affine by the second part, hence so is the composition $\epsilon \circ \Delta$.
Since $\epsilon \times \epsilon$ is separated, its diagonal is a closed immersion and hence affine. By the cancellation
lemma it follows that $\Delta_{\CM^G}$ is affine.
\end{proof}

If $G$ acts on a separated scheme, then the fixed stack is a closed subscheme
and equal to the fixed locus defined in the usual way.
However, in general the map $\epsilon \colon \CM^G \to \CM$ may behave quite subtle.
For example, taking fixed stacks usually does not commute
with passing to the good or coarse moduli space (if it exists).

\subsection{The fixed stack of a $\BG_m$-gerbe}
Consider a $G$-action $(\rho,\theta)$ on the stack
$B \BG_m$
such that $\rho_g = \id$ for all $g \in G$ but we allow the $2$-isomorphisms $\theta$ to be arbitrary.
According to Lemma~\ref{lem:obstruction to action} there is an associated class 
\[ \alpha(\theta) \in H^2(G, \BC^{\ast}) \]
where we let the trivial action correspond to the trivial class.\footnote{We have stated Lemma~\ref{lem:obstruction to action} only for additive $\BC$-linear category,
but since $\Aut( \id_{B \BG_m} ) = \BC^{\ast} \id$ on which $G$ acts trivially by conjugation, the result applies verbatim also in this case.}
We have \cite{R}
\[ (B \BG_m)^G
=
\begin{cases}
\bigsqcup_{\chi \in G^{\vee}} B \BG_m & \text{ if } \alpha(\theta) = 0, \\
\varnothing & \text{ if } \alpha(\theta) \neq 0.
\end{cases}
\]

In this section we consider the following generalization:
Let $M$ be a complete variety, and consider the trivial $\BG_m$-gerbe
\[ \CM = M \times B \BG_m. \]
The projection and the section of the gerbe are denoted by
\[ p_1 \colon \CM \to M, \quad s = (\id_M, t) \colon M \to \CM \]
where $t \colon M \to B \BG_m$ corresponds to the trivial line bundle.
We refer to \cite[Def.\ 12.2.2]{Olsson} for a definition of gerbes and morphisms of gerbes.

\begin{lemma} \label{LemmaGerbe}
There is a natural bijection between the set of morphisms of $\BG_m$-gerbes $f \colon \CM \to \CM$
and the set of pairs $(F, \CL)$ where $F \colon M \to M$ is an automorphism and $\CL \in \Pic(M)$.
If the morphism $f$ corresponds to $(F,\CL)$ and $g$ corresponds to $(G,\CL')$, then $f \circ g$ corresponds to
$(F \circ G, \CL \otimes F^{\ast}(\CL'))$.
\end{lemma}

%

\begin{proof}
See also \cite{Heinloth} for an equivalence on the categorical level.
Let $f \colon \CM \to \CM$ be a morphism of gerbes.
Define
$F = p_1 \circ f \circ s$
and let $\CL$ be the line bundle corresponding to $p_2 \circ f \circ s \colon M \to B \BG_m$.
By \cite[Lem.\ 12.2.4]{Olsson} $F$ is an automorphism. 

Let $L_{\mathrm{univ}}$ be the universal line bundle on $B \BG_m$.
We write $L_{\mathrm{univ}}$ also for its pullback to $M \times B \BG_m$.
Since $f$ is a morphism of gerbes we have\footnote{The restriction to each $m \times B \BG_m$ is equal to $L_{\mathrm{univ}}$ by hypothesis.
Hence $f^{\ast} L_{\mathrm{univ}} = L_{\mathrm{univ}} \otimes p_1^{\ast}L$ for some $L \in \Pic(M)$. Restricting to $M$ yields the claim.}
\[ f^{\ast} L_{\mathrm{univ}} = (f^{\ast} L_{\mathrm{univ}})|_{M} \otimes L_{\mathrm{univ}} = p_1^{\ast}(\CL) \otimes L_{\mathrm{univ}}. \]
Hence given $(F, \CL)$ we can recover $f$ as the product of $F \circ p_1$ and the morphism
associated to $p_1^{\ast}(\CL) \otimes L_{\mathrm{univ}}$.
This yields the 1-to-1 correspondence.

For the last claim, we have that
\[ g^{\ast} L_{\mathrm{univ}} = (g^{\ast} L_{\mathrm{univ}})|_{M} \otimes L_{\mathrm{univ}} = p_1^{\ast}(\CL') \otimes L_{\mathrm{univ}} \]
hence
\[ f^{\ast} g^{\ast} L_{\mathrm{univ}} = p_1^{\ast} F^{\ast}(\CL') \otimes f^{\ast} L_{\mathrm{univ}} \]
which gives the claim by restriction to $M$.
\end{proof}

Let $(\rho, \theta)$ be a $G$-action on $\CM$ such that for all $g \in G$:
\begin{itemize}
\item the morphism $\rho_g$ is a morphism of $\BG_m$-gerbes, and
\item if $(F_g, \CL_g)$ is the pair associated to $\rho_g$, then $F_g = \id$.\footnote{One can always reduce to this case by replacing $\CM$ with $\CM \times_M F$ for an irreducible component $F$ of $M^G$.}
\end{itemize}

For a $\BC$-point $p \in M$ the $G$-action $(\rho, \theta)$ induces an action $(\rho^p, \theta^p)$ on $p \times B \BG_m$
such that for all $g \in G$ we have $\rho^p_g \cong \id_{B \BG_m}$ (since $\rho_g$ acts by gerbe morphisms).
Hence as before we have an associated class
\[ \alpha(\theta^p) \in H^2(G, \BC^{\ast}). \]
The class $\alpha(\theta^p)$ vanishes if and only if $(p \times B \BG_m)^G$ is non-empty.
In this case we say that $p \in M$ is \textit{$G$-linearizable}.

By Remark~\ref{rmk:fixed stack base change} the fixed stack $\CM^G$ is non-empty if and only if $M$ contains a $G$-linearizable point.
Hence let $p \in M$ be $G$-linearizable. The $2$-isomorphisms $\theta_{g,h} \colon \rho_g \rho_h \to \rho_{gh}$ induce isomorphisms
\begin{equation} \theta_{g,h} \colon \CL_g \otimes \CL_h \xrightarrow{\cong} \CL_{gh} \label{eqn:mult} \end{equation}
which satisfy the associativity relations \eqref{associativity}.
In particular, up to isomorphism
the line bundles $\CL_g$ only depend on the conjugacy class $\bar{g}$ of $g$
and we obtain a group homomorphism
\[ G_{\mathrm{ab}} \to \Pic(M),\, \bar{g} \mapsto [\CL_g] \]
where $G_{\mathrm{ab}}$ is the abelianization of $G$, and $[\CL]$ stands for the isomorphism class of a line bundle $\CL$.

\vspace{5pt}
\noindent \textbf{Claim.} The $G$-action on $\CM$ is isomorphic to an action which factors through $G_{\mathrm{ab}}$
and such that the isomorphisms \eqref{eqn:mult} are commutative, i.e.\ $\theta_{g,h} = \theta_{h,g}$
where we identify $\CL_g \otimes \CL_{h}$ with $\CL_{h} \otimes \CL_g$ by swapping the factors.
\begin{proof}[Proof of Claim]
Let $H = [G,G]$ and choose representatives $\{ g_1, \ldots, g_r \}$ for the cosets $G/H$ where we take the identity element for the unit coset.
Given any element $g \in g_i H$ we set
$\rho'_g=\rho_{g_i}$.
The isomorphisms $\CL_g \cong \CL_{g_i}$ induced by \eqref{eqn:mult} yield isomorphisms $t_g \colon \rho_g \cong \rho_{g_i} = \rho'_g$.
Consider the action $(\rho'_g, \theta')$ on $\CM$ where $\theta'$ is determined by the commutative diagram
\[
\begin{tikzcd}
\rho_{g} \rho_h \ar{d}{t_g t_h} \ar{r}{\theta_{g,h}} & \rho_{gh} \ar{d}{t_{gh}} \\
\rho'_g \rho'_h \ar{r}{\theta'_{g,h}} & \rho'_{gh}.
\end{tikzcd}
\]
By construction, $\rho'_g$ only depends on the image of $g$ in $G/H$.
We need to show that we can further modify $\theta'$ such that it also only depends on the image in $G/H$, and is commutative.
The key idea is that since $M$ is a complete variety, $\Hom(\CL_g, \CL_g) = \BC$, and hence we may find and check all the required relations
by restricting to the point $p \in M$ where the action is trivial.
Concretely, we may first choose an identification $\CL_g|_{p} \cong \BC_p$ for every $g$.
Since $\alpha(\theta^p) = 0$ we may then modify $\theta'$ (i.e.\ replace $\theta_{g,h}'$ by $\lambda_{g,h} \theta_{g,h}'$ for some $\lambda_{g,h} \in \BC^{\ast}$
which is the derivative of a $1$-cycle)
such that the restrictions
\[ \theta'_{g,h}|_{p} \colon \CL_g|_{p} \otimes \CL_h|_{p} \to \CL_{gh}|_{p} \]
are the identity maps under the given identification.
Since $\CL_{g}$ only depends on $G/H$ it follows that $\theta_{g,g'}$ only depends on the image of $g$ and $g'$ in $G/H$.
(To spell this out: for any $g \in g_i H$, $g' \in g_j H$
and $h, h' \in H$ we have that $\theta_{g,g'}$ and $\theta_{gh, g'h'}$ are both morphisms $\CL_{g_i} \otimes \CL_{g_j} \to \CL_{g_k}$ where $g_i g_j \in g_k H$;
they agree after restriction to $p$ hence they must agree.)
Similarly, the commutativity $\theta'_{g,g'} = \theta'_{g',g}$ 
follows by restriction. 
\end{proof}

After replacing $(\rho, \theta)$ with an isomorphic action as in the Claim, we obtain a commutative $\CO_M$-algebra
\[ \CA= \bigoplus_{g \in G_{\mathrm{ab}}} \CL_g, \]
where the multiplication is induced by $\theta$.
Consider the \'etale cover
\[ \pi \colon Y \to M, \quad Y = \Spec(\CA). \]
For every $g \in G$ the natural inclusion $\CL_g \to \CA$ yields a natural isomorphism
\begin{equation} \phi_g \colon \pi^{\ast}(\CL_g) \xrightarrow{\cong} \CO_{Y}. \label{lin_phi_g}\end{equation}
The composition
\[ \pi^{\ast}(\CL_g \otimes \CL_h) \xrightarrow{\phi_g \otimes \id_{\CL_h}} \pi^{\ast}( \CL_h ) \xrightarrow{\phi_h} \CO_Y \]
is induced by $\CL_{g} \otimes \CL_h \to \CA \otimes \CA \to \CA$ and hence isomorphic to
\[ \pi^{\ast} (\CL_{g} \otimes \CL_h) \xrightarrow{\pi^\ast\theta_{g,h}} \pi^{\ast} \CL_{gh} \xrightarrow{\phi_{gh}} \CO_Y. \]
We see that $\phi_g$ gives $s \circ \pi \colon Y \to \CM$ the structure of a $G$-equivariant morphism with respect to the trivial action on $Y$.
This yields a morphism $Y \to \CM^G$.

Define the product 
\[ \CY = Y \times B \BG_m \]
and consider the morphism
\[ f = \pi \times \id_{B \BG_m}\colon \CY \to \CM. \]
As before, the tensor product of $\phi_g$ with the identity on the universal bundle makes $f$ equivariant with respect to the trivial action on $\CY$.
We obtain a morphism $\CY \to \CM^G$.
This yields the following description of the fixed stack. 

\begin{prop} \label{PropGerbe} In the setting above, if $M$ contains a $G$-linearizable point, then
$f \colon \CY \to \CM$ is the fixed stack of the $G$-action on $\CM$.
\end{prop}
\begin{proof}
We have seen above that there is a natural morphism $\CY \to \CM^G$. 
Conversely, giving an equivariant morphism $h \colon T \to M \times B \BG_m$,
where the scheme $T$ carries the trivial $G$-action,
is equivalent to a line bundle $L$ on $T$, a morphism $h' = p_1 \circ h \colon T \to M$ and
maps $h^{\prime \ast} \CL_{g} \to \CO_T$ satisfying the cocycle condition.
The cocycle condition implies that the induced map
\[ h^{\prime \ast}( \oplus_{g \in G_{\mathrm{ab}}} \CL_g ) \to \CO_T \]
is an algebra homomorphism with respect to the algebra structure on $\oplus_g \CL_g$ defined by $\theta$.
Hence the map $T \to M$ factors through $Y$
and thus $h$ factors through $Y \times B \BG_m$.
This yields the inverse $\CM^G \to \CY$.
\end{proof}

\begin{rmk} \label{rmk_non-trivial-gm-gerbe}
Parallel results hold for a non-trivial $\BG_m$-gerbe $\pi \colon \CM \to M$ with Brauer class $\alpha \in \mathrm{Br}(M)$:
There exists a $\pi^{\ast}(\alpha)$-twisted line bundle $L_{\mathrm{univ}}$ on $\CM$ (playing the role of $L_{\mathrm{univ}}$ as before)
such that for every morphism $f \colon X \to M$ and every $f^{\ast}(\alpha)$-twisted line bundle $\CL$ on $X$
there exists a unique map $F \colon X \to \CM$ such that $F^{\ast}(L_{\mathrm{univ}}) = \CL$ and $f = \pi \circ F$.
A morphism $F \colon \CM\to \CM$ of $\BG_m$-gerbes is then equivalent to the pair of an (untwisted) line bundle $L$ on $M$
and a morphism $f \colon M \to M$ such that $f^{\ast}(\alpha) = \alpha$.
See also \cite[Sec.\ 5]{Heinloth}.
The formulation of the analogue of Proposition~\ref{PropGerbe} is similar.
\end{rmk}

\subsection{Moduli spaces of equivariant objects} \label{subsec:moduli_of_equ}
Let $X$ be a smooth projective variety over $\BC$.
Recall from \cite{Lieblich} the stack 
\[ \mathfrak{M} \colon \mathrm{Sch}/\BC \to \Grpds \]
which associates to each scheme $T$
the groupoid of $T$-perfect universally gluable objects in $D(X \times T)$.
As proven in \emph{loc.~cit.} $\mathfrak{M}$ is a quasi-separated algebraic stack locally of finite type over $\BC$
with affine diagonal, see also \cite[0DPV]{StacksProject} and \cite[Sec.\ 8]{StabinFamilies}. 

Let $G$ be a finite group which acts on $D^b(X)$.
By Lemma~\ref{lem:FM-action} the action is given by Fourier--Mukai transforms.
The pullback of the Fourier--Mukai kernels define a Fourier--Mukai action $D(X \times T)$
such that the pullback morphisms are $G$-equivariant.
This defines an action of $G$ on $\mathfrak{M}$ in the sense of Section~\ref{Subsection_group_action_on_Stacks},
\[ (\rho, \theta) \colon G \times \mathfrak{M} \to \mathfrak{M}. \]
Remark~\ref{rmk_fixed_stack} yields the following description of the fixed stack:

\begin{prop} \label{prop_fixed_stack}
The fixed stack $\mathfrak{M}^G$ is the stack of $G$-equivariant
universally gluable perfect complexes in $D(X)$, i.e.\ for every scheme $T$ we have
\[
\mathfrak{M}^G(T)
= \{ (\CE,\phi) \in D(X \times T)_{G \times 1} \mid \CE \text{ is universally gluable, $T$-perfect} \}.
\]
The isomorphisms in $\mathfrak{M}^G(T)$ are the isomorphisms of objects in $D(X \times T)_{G \times 1}$. The pullback is the equivariant pullback.
The morphism $\epsilon \colon \mathfrak{M}^G \to \mathfrak{M}$ is the map that forgets the $G$-linearization.
\end{prop}

From now on let $\sigma$ be a stability condition on $D^b(X)$ which is preserved by the $G$-action.
Let $\CM_{\sigma}(v)$ be the moduli stack of $\sigma$-semistable objects of class $v \in K(\CA)$,
i.e.\ for any scheme $T$ we let
\[
\CM_{\sigma}(v)(T)
= \{ \CE \in D(X \times T) \mid \forall t \in T\colon \CE_{t} \text{ is $\sigma$-semistable with } [E_t]=v \}.
\]
Since $G$ preserves $\sigma$-semistability,
for any $G$-invariant $v \in K(\CA)$ we have an action
\[ G \times \CM_{\sigma}(v) \to \CM_{\sigma}(v). \]
The following result follows immediately from 
Proposition~\ref{prop_fixed_stack} and Lemma~\ref{Lemma_stability}. 
\begin{prop} \label{prop_fixed_stack_Semistable}
We have
\[ \CM_{\sigma}(v)^G = \bigsqcup_{\substack{v' \in K(\CA_G) \\ p_{\ast}(v') = v}} \CM_{\sigma_G}(v'), \]
where $\CM_{\sigma_G}(v')$ is the substack of $\mathfrak{M}^G$ defined by
\[ \CM_{\sigma_G}(v')(T) = \{ \CE \in D(X \times T)_{G \times 1} \mid \forall t \in T\colon \CE_t \text{ is $\sigma_G$-semistable, } [\CE_t] = v' \}. \]
\end{prop}

\subsection{The fixed stack of a fine moduli space}
In the setting of Section~\ref{subsec:moduli_of_equ},
let $v \in K(D^b(X))$ be a $G$-invariant class
such that $\CM_{\sigma}(v)$ has a fine moduli space $M_{\sigma}(v)$ which is smooth.
The goal of this section is to determine the fixed stack $\CM_{\sigma}(v)^G$.

Write $\CM= \CM_{\sigma}(v)$ and $M = M_{\sigma}(v)$.
By assumption there is a universal family
\[ \CE \in D( M \times X), \]
unique up to tensoring with a line bundle pulled back from the first factor.
By the universal property of $\CM$ this yields a section
$s_{\CE} \colon M \to \CM$
of the $\BG_m$-gerbe
$\CM \to M$.
Hence $s_{\CE}$ defines a trivialization
\begin{equation} \CM \cong M \times B\BG_m. \label{1992} \end{equation}
The universal family $\CE_{\CM} \in D(\CM \times X)$
is identified under \eqref{1992} with
\[ 
(p_1 \times \id_X)^{\ast}(\CE) \otimes p_2^{\ast}(L_{\mathrm{univ}}) 
\]
where $p_1, p_2$ are the projections to the factors.

Let $f \colon \CM \to \CM$ be a morphism of $\BG_m$-gerbes and let
\[ F = p_1 \circ f \circ s_{\CE}, \quad \CL = (p_2 \circ f \circ s_{\CE})^{\ast} L_{\mathrm{univ}} \]
be the associated automorphism and line bundle as in Lemma~\ref{LemmaGerbe}.
We consider the difference of the pullbacks of the universal families under $F$ and $f$.

\begin{lemma} \label{LemmaUnivFam}
In the situation above, we have
\[ ((f\times \id_X)^{\ast}\CE_{\CM})|_{M} = (F\times \id_X)^{\ast}(\CE) \otimes \CL. \]
\end{lemma}

\begin{proof}
Under the identification \eqref{1992} 
we have $\CE_{\CM} = (p_1\times \id_X)^{\ast}(\CE) \otimes L_{\mathrm{univ}}$.
Hence
\begin{align*} (f\times \id_X)^{\ast}( \CE_{\CM} ) & = (f\times \id_X)^{\ast}( (p_1\times \id_X)^{\ast}(\CE) ) \otimes (f\times \id_X)^{\ast} L_{\mathrm{univ}}  \\
& = (p_1\times \id_X)^{\ast}( (F\times \id_X)^{\ast}(\CE) ) \otimes ( (p_1 \times \id_X)^{\ast}(\CL) \otimes  L_{\mathrm{univ}} ) \\
& = (p_1\times \id_X)^{\ast}( (F\times \id_X)^{\ast}(\CE) \otimes \CL) \otimes  L_{\mathrm{univ}}.
\end{align*} 
Restricting to $M$ completes the claim.
\end{proof}

Consider the action of $G$ on $\CM$. 
For every $g \in G$ the morphism $\rho_g \colon \CM \to \CM$
commutes with the inclusion of the automorphism groups (in the derived category, we have $g (\lambda \id ) = \lambda g (\id) = \lambda \id$)
and hence is a morphism of $\BG_m$-gerbes.
Let
\[ F_g \colon M \to M, \quad \CL_g \in \Pic(M) \]
be the associated pair constructed in Lemma~\ref{LemmaGerbe}.
By Lemma~\ref{LemmaUnivFam} the line bundle $\CL_g$ can also be described by
\begin{equation} (1 \times g)(\CE) = ((1 \times g)\CE_{\CM})|_{M} 
= ((\rho_g \times \id_X)^{\ast}\CE_{\CM})|_M = (F_g\times \id_X)^{\ast}(\CE) \otimes \CL_g. \label{dsfdsf} \end{equation}

Let $F$ be a connected component of the fixed locus $M^G \subset M$ and let
$L_g = \CL_g|_{F}$
which only depends on the conjugacy class of $g$,
see the discussion in Section~\ref{Subsection_group_action_on_Stacks}.
Consider further the associated \'etale cover
\begin{equation} Y = \Spec\left( \bigoplus_{g \in G_{ab}} L_{g} \right), \quad \pi \colon Y \to F \label{defn:etale cover 2} \end{equation}
and define
\[
\CY = Y \times B \BG_m, \quad \epsilon \colon \CY \xrightarrow{\pi \times \id_{B \BG_m}} F \times B \BG_m \to \CM.
\]

\begin{prop} \label{prop_fixed_stack_for_fine_moduli_space}
In the setting above, 
if $F$ contains a $G$-linearizable point,
then $\CY$ 
is the union of the connected components of $\CM^G$ which map to $F$
and $\epsilon \colon \CY \to \CM$ is the restriction of the classifying map $\CM^G \to \CM$ to $\CY$.
The universal linearization of $\epsilon^{\ast}( \CE_{\CM} )$
is pulled back from the canonical linearization of $(\pi \times \id_X)^{\ast}( \CE|_{F \times X})$.
\end{prop}

By Proposition~\ref{prop_fixed_stack}, a point $p \in F$ is $G$-linearizable
if and only if the corresponding $G$-invariant object $\CE_p$
is $G$-linearizable.
Using Proposition~\ref{prop_fixed_stack_for_fine_moduli_space}
we see that there exists a $G$-linearizable point $p \in F$
if and only if every point on $F$ is $G$-linearizable.
In this case we say that the connected component $F$ of $M^G$ is $G$-linearizable.

\begin{proof}
The first statement is Proposition~\ref{PropGerbe}.
The second part follows since the linearization on $\CY$ is the pullback
of the linearization on $Y$ given by \eqref{lin_phi_g}.
\end{proof}

\begin{rmk} \label{remark_Gdual_action}
The action of $G^{\vee}$ on $D^b(X)_G$ by twisting the linearization preserves the stability condition $\sigma_G$.
Moreover, for every $\chi \in G^{\vee}$ we have $p_{\ast} \chi v' = p_{\ast} v'$.
Hence we have an induced action of $G^{\vee}$ on
\[ \CM_{\sigma}(v)^G = \bigsqcup_{p_{\ast}(v') = v} \CM_{\sigma_G}(v'). \]
By Lemma~\ref{lem:Ploog} the action is free if $\CM_{\sigma}(v)$ is a moduli space of stable objects.

In the setting of Proposition~\ref{prop_fixed_stack_for_fine_moduli_space},
the $G^{\vee}$-action can be desribed by letting a character $\chi \in G^{\vee}$ act on the line bundle $\CL_g$ by multiplication by $\chi(g)$.
In particular, $Y/G^{\vee} = F$ and the projection
$\pi \colon Y \to F$
is a $G^{\vee}$-torsor (in the \'etale topology).
\end{rmk}

\begin{rmk} \label{rmk:coarse moduli space}
The results of this section can be generalized to the case when $\pi \colon \CM_{\sigma}(v) \to M_{\sigma}(v)$ is a non-trivial $\BG_m$-gerbe
(if $\CE \in D(M \times X, -\alpha)$ is the twisted universal object, then
the universal family $\CE_{\CM}$ on the stack $\CM \times X$ is given by $\pi^{\ast}(\CE) \otimes L_{\mathrm{univ}}$,
see also Remark~\ref{rmk_non-trivial-gm-gerbe}).
\end{rmk}

\begin{example} \label{Example_EE}
Let $E$ be an elliptic curve and let $t_a \colon E \to E$ be the translation by a $2$-torsion point $a \in E$.
The group $G = \BZ_2$ acts on $\Coh(E)$ by $t_a^{\ast}$.
Let $E' = E/t_a$.
The equivariant category is
$\Coh(E)_{G} = \Coh(E')$.
Consider the moduli stack $\CM = \CM(1,0)$ of Gieseker stable sheaves with Chern characters $v=(1,0) \in H^{2 \ast}(E)$
or equivalently the moduli stack of degree $0$ line bundles. It admits the fine moduli space
$M \cong E$ with universal family the Poincar\'e bundle $\CP$ on $E \times E$. Hence
$\CM \cong E \times B \BG_m$.
Since every degree $0$ line bundle is translation invariant,
the group $G$ induces the trivial action on $M$. However, because of
\[ (1 \times t_a^{\ast})(\CP) = (\id \times t_a)^{\ast} \CP = \CP \otimes p_1^{\ast} \CP_{a}, \]
the bundle $\CP$ can not be linearized over $M$.
Indeed by Proposition~\ref{prop_fixed_stack_for_fine_moduli_space} (with $L_g = \CP_a$) one has
$\CM^G = \tilde{E} \times B \BG_m$ where $\tilde{E}$ is the cover of $E$ defined by $\CP_a$. 

An alternative description of the fixed stack is also given by Proposition~\ref{prop_fixed_stack_Semistable}
as follows:
\[ \CM^G = \CM_{E'}(1,0) \cong E' \times B \BG_m. \]
Since $E' \cong \tilde{E}$ these two presentations agree with each other.
\end{example}

\begin{example} \label{example:no G-lin}
We give an example of a component which is not $G$-linearizable.

Let $G = \BZ_2 \times \BZ_2$ be the subgroup of $2$-torsion points of $E$ acting by translation.
Let $\CM = \CM(1,2)$ be the moduli stack of degree $2$ line bundles and let $M \cong E$ be its fine moduli space.
Then $M^G = M$ but $\CM^G = \varnothing$, so $M$ is not $G$-linearizable. Indeed,
any $G$-linearization of a degree $2$ line bundle $L$ is a descent datum for the quotient map $\pi \colon E \to E/G$.
Hence there would exists a line bundle $L'$ on $E/G$ with $\pi^{\ast} L' = L$ which would imply that the degree of $L$ is divisible by $4$.
\end{example}

\subsection{The Artin--Zhang functor}
\label{subsec:ArtinZhang}
As before we consider an action of a finite group $G$ on $D^b(X)$ which preserves a stability condition $\sigma = (\CA,Z)$.
In this section we further assume the following properties:
\begin{itemize}
\item $\CA$ is Noetherian,
\item $\CA$ satisfies the generic flatness property of \cite[Prop.\ 3.5.1]{AP}.
\end{itemize}
The second condition implies that the subfunctor $\mathfrak{M}_{\CA} \subset \mathfrak{M}$ of objects, such that every geometric fiber lies in $\mathfrak{\CA}$, is open.
By Remark~\ref{rmk:fixed stack base change} the open immersion $\mathfrak{M}_{\CA} \subset \mathfrak{M}$ yields the fiber diagram
\begin{equation}
\label{sdasd}
\begin{tikzcd}
(\mathfrak{M}_{\CA})^G \ar[hookrightarrow]{r} \ar{d}{\epsilon} & \mathfrak{M}^G \ar{d} \\
\mathfrak{M}_{\CA} \ar[hookrightarrow]{r} & \mathfrak{M}.
\end{tikzcd}
\end{equation}
By base change this shows that also $(\mathfrak{M}_{\CA})^G \subset \mathfrak{M}^G$ is an open immersion.

Given a cocomplete\footnote{i.e.\ $\CC$ has all small filtered colimits}, locally noetherian, $k$-linear abelian category $\CC$,
let $\CN_{\CC}$ be the stack of finitely presented objects in $\CC$
as introduced by Artin and Zhang \cite{AZ}, see also \cite[Def.\ 7.8]{AHLH}.
Concretely, for a commutative ring $R$ let $\CC_{R}$ be the category of pairs $(E,\phi)$ with $E$ an object in $\CC$ and $\phi \colon R \to \End_{\CC}(E)$ a morphism of $k$-algebras.
Then $\CN_{\CC}(\Spec R)$ is the groupoid of flat and finitely presented objects in $\CC_{R}$,

As discussed in \cite[Ex.\ 7.20]{AHLH} our assumptions on $\CA$ imply that the stacks $\mathfrak{M}_{\CA}$ and $\CN_{\mathrm{Ind}(\CA)}$ are equivalent,
where $\mathrm{Ind}(\CA)$ is the Ind-completion of $\CA$.
Our first goal is to prove the parallel result for the equivariant abelian category $\CA_G$:
\begin{prop} \label{prop:M=N} $(\mathfrak{M}_{\CA})^G \cong \CN_{\mathrm{Ind}(\CA_G)}$. \end{prop}

We begin with two technical lemmata.
\begin{lemma}
If $\CA$ is a Noetherian abelian $\BC$-linear category, then every object in $\Ind(\CA)$ can be written as a union of objects in $\CA$.
\end{lemma}
\begin{proof}\footnote{We thank Eugen Hellman for providing this argument.}
Given objects $E \in \CA$ and $F \in \Ind(\CA)$ and an inclusion $F \subset E$ in $\Ind(\CA)$
we first claim that $F \in \CA$.
Indeed, write $F = \lim_i F_i$ where the $F_i$ lie in $\CA$. Then since $F \to E$ is a monomorphism we have $F_i' \coloneqq \mathrm{Im}(F_i \to F) = \mathrm{Im}(F_i \to E)$ and thus this image lies in $\CA$. Therefore, $F$ is a union of objects in $\CA$ (namely the $F_i'$) which are subjects of $E$.
Since $E$ is Noetherian, this union has to stabilize and since abelian categories contain finite colomits, $F \in \CA$ as desired.
Now, if $E \to F$ is a quotient in $\Ind(\CA)$ with $E \in \CA$ and $F \in \Ind(\CA)$ then by the above the kernel lies in $\CA$ and hence so does $F$.
Therefore $\CA$ is closed under quotients in $\Ind(\CA)$.
We conclude, that if $E = \lim_{i} E_i$ with $E_i \in \CA$, then $E$ is the union of the $F_i = \mathrm{Im}(E_i \to E)$.
\end{proof}

\begin{lemma} \label{indG=Gind}
Let $\CA$ be a Noetherian abelian $\BC$-linear category and $G$ a finite group. Then there exists a canonical isomorphism $\Ind(\CA_G) \cong \Ind(\CA)_G$.
\end{lemma}

We refer to \cite[Lem.\ 3.6]{Perry} for a parallel result for $\infty$-categories.

\begin{proof}
If $\CA$ is cocomplete
and $(E_i, \phi_i)$ is a direct system in $\CA_G$,
then the $\phi_i$ define a canonical $G$-linearization on
$E = \lim E_i$. Hence $\CA_G$ is also cocomplete.

Let $\CA$ now be Noetherian. Applying the above argument to $\Ind(\CA)$ we see that $\Ind(\CA)_G$ is cocomplete.
Hence by the universal property of Ind-completion,
the inclusion $\CA_G \to \Ind(\CA)_G$ lifts
to a functor $\Ind(\CA_G) \to \Ind(\CA)_G$.
By composing with the forgetful functor $\Ind(\CA)_G \to \Ind(\CA)$ one sees the functor is faithful.
We check that the functor is essentially surjective and full:
Let $(E,\phi) \in \Ind(\CA)_G$ where $E = \bigcup_{i} E_i$ is a union of objects $E_i$ in $\CA$.
By replacing $E_i$ by $\bigcup_{g \in G} \phi_g^{-1}( g E_i )$ if necessary we get that the restrictions $\phi_g|_{E_i} \colon E_i \to g E_i$
define $G$-linearizations on $E_i$. Moreover, after replacing the $E_i$ and $F_i$ suitably, any morphism $(E,\phi) \to (F,\psi)$ is
the limit of a morphism $(E_i, \phi_i) \to (F_i, \psi_i)$.
\end{proof}

\begin{proof}[Proof of Proposition~\ref{prop:M=N}]
Since $\mathfrak{M}_{\CA} = \CN_{\mathrm{Ind}(\CA)}$
we have that $\mathfrak{M}_{\CA}^G(\Spec R)$ is the groupoid of pairs of $x \in \CN_{\CA}(R)$ together with linearizations $\phi_g \colon x \to gx$ satisfying the cocycle condition.
Spelling this out this is the groupoid of triples of objects $E \in \Ind(\CA)$, homomorphisms $\sigma \colon R \to \End(E)$ and linearizations $\phi_g \colon E \to gE$ satisfying
\[ \phi_g \circ \sigma_r = g \sigma_r \circ \phi_g, \]
or equivalently, 
the groupoid of pairs $(E,\phi) \in \Ind(\CA)_G$ and $\sigma \colon R \to \End_{\Ind(\CA)_G}(E,\phi)$.
However, $G$ finite implies that $\Ind(\CA)_G = \Ind(\CA_G)$ (see Lemma~\ref{indG=Gind})
and hence this is precisely the groupoid $\CN_{\mathrm{Ind}(\CA_G)}(\Spec R)$.
\end{proof}

A stability condition $\sigma = (\CA,Z)$ is called \textit{algebraic} if $Z(K(\CA)) \subset \BQ+ i \BQ$. 

\begin{thm} \label{thm:existence_gms}
In the above situation assume moreover that
$\sigma$ is algebraic and that  $\CM_{\sigma}(v)$ is bounded for every $v \in K(D(X))$.
Then for every $v' \in K(D^b(X)_G)$ the moduli stack $\CM_{\sigma_G}(v')$
is an universally closed Artin stack of finite type over $\BC$ which
has a proper good moduli space. The inclusion
$\CM_{\sigma_G}(v') \to \mathfrak{M}^G$
is an open embedding.
\end{thm}
\begin{proof}
Let $v = p_{\ast} v'$ and let $\mathfrak{M}_{\CA, v} \subset \mathfrak{M}_{\CA}$ be the open and closed substack parametrizing objects of class $v$.
Invoking \cite[Ex.\ 7.27]{AHLH}, the stack $\mathfrak{M}_{\CA,v}$ has a $\Theta$-stratification whose open piece is $\CM_{\sigma}(v)$.
This yields the fiber diagram
\[
\begin{tikzcd}
\CM_{\sigma}(v)^G \ar[hookrightarrow]{r} \ar{d}{\epsilon} & (\mathfrak{M}_{\CA,v})^G\ar{d}{\epsilon} \\
\CM_{\sigma}(v) \ar[hookrightarrow]{r} & \mathfrak{M}_{\CA,v},
\end{tikzcd}
\]
where the horizontal maps are open immersions.
Since $\mathfrak{M}_{\CA,v} \subset \mathfrak{M}$ is open
and $\mathfrak{M}$ is an Artin stack locally of finite type with affine diagonal over $\BC$,
applying Proposition~\ref{algebraicity_of_fixed_stack} the same holds for $(\mathfrak{M}_{\CA,v})^G$. Moreover, by Proposition~\ref{algebraicity_of_fixed_stack} again both vertical morphisms $\epsilon$ are affine.
Since $\CM_{\sigma}(v)$ is of finite type, so is $\CM_{\sigma}(v)^G$.

By \cite[Sec.\ 7]{AHLH} the stack $\CM_{\sigma}(v)$ is $\Theta$-reductive and S-complete.
By \cite[Prop.\ 3.20(1)]{AHLH} affine morphisms are $\Theta$-reductive and by \cite[Prop.\ 3.42(1)]{AHLH} they are $S$-complete.
Since both these properties are stable under composition, $\CM_{\sigma}(v)^G$ is $\Theta$-reductive and $S$-complete
and hence by \cite[Thm.\ A]{AHLH} admits a separated good moduli space.

It remains to show that $\CM_{\sigma}(v)^G$ is universally closed.\footnote{
Since $\epsilon$ is not proper in general (see Section~\ref{subsec:dual action} for an example where this fails) this does not follow directly from the fact that $\CM_{\sigma}(v)$ is universally closed.
Instead we use the alternative description of the bigger stack $(\mathfrak{M}_{\CA})^G$.}
For this recall from Proposition~\ref{prop:M=N} the isomorphism $(\mathfrak{M}_{\CA})^G \cong \CN_{\mathrm{Ind}(\CA)}$.
It follows from \cite[Lem.\ 7.17]{AHLH} that $\mathfrak{M}_{\CA}^G$ satisfies the existence
part of the valuative criterion of properness.
Since $\epsilon \colon (\mathfrak{M}_{\CA,v})^G \to \mathfrak{M}_{\CA,v}$ is affine,
by \cite[Prop.\ 1.19]{DHL1} the preimage of the $\Theta$-stratification of $\mathfrak{M}_{\CA,v}$ defines a $\Theta$-stratification of $(\mathfrak{M}_{\CA,v})^G$.
By definition its open piece is the preimage of the stack of $\sigma$-semistable objets, which, 
is precisely the stack of $\sigma_G$-semistable objects.\footnote{
The $\Theta$-stratification of $\mathfrak{M}_{\CA,v}$ corresponds to the Harder--Narasimhan filtration in $\CA$.
Given an equivariant object $(E,\phi)$ and a Harder--Narasimhan filtration $E_i$ of $E$ with respect to $\sigma$ the restrictions $(E_i, \phi|_{E_i})$ define a Harder--Narasimhan filtration of $(E,\phi)$
which corresponds to the 'preimage' $\Theta$-stratification of $(\mathfrak{M}_{\CA})^G$.}
By semistable reduction \cite[Thm.\ B/C]{AHLH} we conclude that $\CM_{\sigma}(v)^G$ is universally closed and therefore that its good moduli space is proper.
By Proposition~\ref{prop_fixed_stack_Semistable} the stack $\CM_{\sigma_G}(v')$ is a closed and open substack of $\CM_{\sigma}(v)^G$,
hence it satisfies the same conclusion.
\end{proof}

We consider the deformation-obstruction theory of the functor $\mathfrak{M}_{\CA}^G$.
\begin{prop} \label{def_obst_thy}
Suppose that $\CA$ is Noetherian, satisfies the generic flatness property
and we have $D^b(\CA) \cong D^b(X)$.

Let $0 \to I \to A' \to A \to 0$ be a square zero extension of rings and let 
$\iota \colon X \times \Spec A \to X \times \Spec A'$
be the natural inclusion.
Let $(E,\phi) \in \mathfrak{M}_{\CA}^G(\Spec A)$.
Then there exists an obstruction class
\[ \omega(E,\phi) \in \Ext^2(E,E \otimes I)^G_0 \]
which vanishes if and only if there exists a complex $(E', \phi') \in \mathfrak{M}_{\CA}^G(A')$
such that $\iota^{\ast}(E',\phi') \cong (E,\phi)$.
Moreover, in this case the set of extensions is a torsor over $\Ext^1(E,E \otimes I)^G$.
\end{prop}
Here the subscript $0$ stands for the traceless part defined by
\[ \Ext^2(E,E)_0 = \mathrm{Ker} \left( \mathrm{Tr} \colon \Ext^2(E,E) \to H^2(X,\CO_X) \right). \]

\begin{proof}
By Proposition~\ref{prop:M=N} we can use the deformation theory of the Artin--Zhang functor $\CN_{\mathrm{Ind}(\CA)}$.
Since $D^b(\CA) = D^b(X)$ for any $(E,\phi) \in \CA_G$ we have
\[ \Ext^i_{D^b(\CA_G)}( (E,\phi), (E,\phi) ) = \Ext^i_{D^b(X)_G}( (E,\phi), (E,\phi) ) = \Ext^i_{D^b(X)}(E,E)^G. \]
Hence the existence of the obstruction class $\omega(E,\phi) \in \Ext^2(E, E \otimes I)^G$ follows from \cite{Lowen}.
The ($G$-invariant) trace map is the derivative to the determinant map on $S$.
Since the Picard stack is smooth, all obstructions to deforming $\det(E)$ vanishes.
This shows that the obstruction class lies in the kernel of
\[ \Ext^2(E,E)^G \xrightarrow{p_{\ast}} \Ext^2(E,E) \xrightarrow{\mathrm{Tr}} \BC. \qedhere \]
\end{proof}

\subsection{Conclusion} \label{sec:summary}
Let $X$ be a smooth projective variety and let
$\Stab^\ast(X) \subset \Stab(X)$
be a connected component of the stability manifold satisfying the following condition:

\begin{enumerate}
\item[($\dagger$)] There exists an algebraic stability conditions $\sigma = (\CA,Z) \in \Stab^{\ast}(X)$ such that 
\begin{itemize}
\item $\CA$ satisfies the generic flatness property and
\item for all $v \in K(\CA)$ the stack $\CM_\sigma(v)$ is bounded.
\end{itemize}
\end{enumerate}

Then by \cite[Prop.\ 4.12]{PT} the same holds for all algebraic stability conditions in $\Stab^\ast(X)$.
Moreover, as explained in \cite[Ex.\ 7.27]{AHLH}, for any $v \in K(D^b(X))$ and stability condition $\sigma \in \Stab^{\ast}(X)$
one can find an algebraic stability condition $\sigma'$ such that $\CM_{\sigma}(v)$ and $\CM_{\sigma'}(v)$ define the same moduli functor.

Assume as before that we have a $G$-action on $D^b(X)$.
We will need the following $G$-invariant version of the argument in \cite[Ex.\ 7.27]{AHLH}.

\begin{lemma}\label{lem:StabCondtoAlgebraicinGinvariant}
Let $v\in K(D^b(X))^G$ and $\sigma \in \Stab^\ast(X)^G$.
Then there exists an algebraic stability condition $\sigma'\in \Stab^\ast(X)^G$, such that $\CM_{\sigma}(v)$ and $\CM_{\sigma'}(v)$ define the same moduli functor.
\end{lemma}
\begin{proof}
We follow the arguments and notations from \cite[Ex.\ 7.27]{AHLH}. Note also that the arguments from \cite[Lem.\ 2.15]{MMS} apply in our setting. 
We restrict the decomposition of \cite{AHLH}
\[
\CC_{S'}=\left( \bigcup_{\gamma' \in S'} \CW_{\gamma'} \right) \setminus \bigcup_{\gamma' \not \in S'} \CW_{\gamma'}
\]
associated to $v$ and $\sigma$ to the set of invariant stability conditions $\Stab^\ast(X)^G$.
Since we have $\sigma \in \CC_{S'}$, we conclude
for all $\gamma' \not \in S'$ that the connected component of the submanifold $\Stab^\ast(X)^G$ containing $\sigma$ is not entirely contained in $\CW_{\gamma'}$.
Then arguing as in \cite[Ex.\ 7.27]{AHLH} for $\CC_{S'} \cap \Stab^*(X)^G$ completes the proof.
\end{proof}

This yields the following existence result.
\begin{thm} \label{thm:existence gms}
Let $\sigma \in \Stab^{\ast}(X)$ be a $G$-fixed stability condition. Then for 
every $v' \in K(D^b(X)_G)$ the stack $\CM_{\sigma_G}(v')$
is a universally closed Artin stack of finite type over $\BC$ which
has a proper good moduli space.
\end{thm}
\begin{proof}
By Lemma~\ref{lem:StabCondtoAlgebraicinGinvariant} we may assume that $\sigma$ is algebraic
and apply Theorem~\ref{thm:existence_gms}.
\end{proof}

We are ready to give a proof of Theorem~\ref{thm:summary}.
\begin{proof}[Proof of Theorem~\ref{thm:summary}]
We will assume for simplicity that $M$ is a fine moduli space.
The case of a coarse moduli space of stable objects works parallel by using a twisted universal object instead, see Remark~\ref{rmk:coarse moduli space}.
By Proposition~\ref{prop_fixed_stack_Semistable} we have the decomposition 
\begin{equation} 
\CM_{\sigma}(v)^G = \bigsqcup_{p_{\ast} v' = v} \CM_{\sigma_G}(v'). \label{bbbbbvcc} 
\end{equation}
The map~\eqref{thm:summary:map} is induced from $\epsilon \colon \CM_{\sigma}(v)^G \to \CM_{\sigma}(v)$ by passing to good moduli spaces.
For every $G$-linearizable connected component $F \subset M^G$,
the scheme
$Y=\Spec\left( \oplus_{g \in G_{\mathrm{ab}}} \CL_g \right)$
as defined in \eqref{defn:etale cover 2}
is a $G^{\vee}$-torsor over $F$, see Remark~\ref{remark_Gdual_action}.
By Proposition~\ref{prop_fixed_stack_for_fine_moduli_space} 
the gerbe $Y \times  B \BG_m $ is the union of all connected components of \eqref{bbbbbvcc} mapping to $F$.
Since every connected component maps to some $F$ this shows the first claim.

If $G$ factors through a Schur cover $G \to Q$, then we have $M^G = M^Q$.
Moreover for every connected component $F$ and point $p \in F$ the obstruction of being $G$-linearizable (as given by Lemma~\ref{lem:Ploog}) is
the pullback of a class in $H^2(Q,\BC^{\ast})$ and hence vanishes.
This shows that every connected component of $M^G$ is $G$-linearizable
and so \eqref{thm:summary:map} is surjective.
\end{proof}

{\large{\part{{Equivariant categories of symplectic surfaces}}}}

\section{More on equivariant categories} \label{sec:more}

\subsection{Calabi--Yau categories} \label{subsec:calabi-yau}
The main reference for this section is \cite{Notes}.

Let $\CD$ be a $\BC$-linear triangulated category with finite-dimen\-sional $\Hom$ spaces.
A \emph{Serre functor} for $\CD$ is an equivalence $S \colon \CD \to \CD$ together with
bifunctorial isomorphisms
\[ \eta_{A,B} \colon \Hom(A,B) \xrightarrow{\cong} \Hom(B, SA)^{\vee} \]
for all objects $A,B \in \CD$.
By \cite[Sec.\ 5]{Notes} if we are given an action by a finite group $G$ on $\CD$ the Serre functor $S$ lifts to a Serre functor 
\[ \tilde{S} \colon \CD_G \to \CD_G \]
which is of the form $\tilde{S}(A,\phi) = (SA, \phi')$ for a certain linearization $\phi'$. 
Moreover, for any objects $(A,\phi)$ and $(B, \psi)$ in $\CD_G$ the restriction of $\eta_{A,B}$ to the $G$-invariant part defines
bifunctorial isomorphisms
\[ \eta_{A,B} \colon \Hom(A,B)^G \xrightarrow{\cong} ( \Hom( B, SA )^G )^{\vee} \]
where the $G$-action on the left is defined by the linearizations $\phi, \psi$
and the $G$-action on the right is defined by the linearizations $\psi$ and $\phi'$ 

We say that the category $\CD$ is Calabi--Yau if there exists a $2$-isomorphism
\[ \id_{\CD} \xrightarrow{\cong} S[-n] \]
for some integer $n$, called the dimension of $\CD$.

\begin{rmk} \label{rmk_serrefunc}
The derived category  $D^b(X)$ of a smooth projective $n$-dimensional variety $X$
has the Serre functor $S = ( - ) \otimes \omega_X[n]$.
In this case we will usually denote the lifted functor $\tilde{S}$ also by $( - ) \otimes \omega_X[n]$
where the action on the linearization is implicitly understood. So
\[ (A,\phi) \otimes \omega_X[n] \]
will stand for $\tilde{S} (A, \phi) = (A \otimes \omega_X[n], \phi')$.
\end{rmk}
\begin{rmk}
The results discussed above work also in the relative case
of a smooth projective morphism $\pi \colon X \to T$ with geometrically connected fibers as in Section~\ref{Subsection_Geometric_case}.
Given a Fourier--Mukai $G$-action on $D(X)$, the $\pi$-relative Serre functor lifts to a $\pi$-relative Serre functor of the equivariant category $D(X)_G$.
\end{rmk}

We have the following criterion for the equivariant category of a Calabi--Yau variety to be Calabi--Yau.

\begin{prop} \textup{(}\cite[Sec.\ 6.3, 6.4]{Notes}\textup{)} \label{prop:calabi-yau}
Let $X$ be a smooth projective variety which is Calabi--Yau, i.e.\ $\omega_X \cong \CO_X$.
Consider the action of a finite group $G$ on $D^b(X)$ which lifts to an action on the dg-enhancement $D_{\mathrm{dg}}(X)$.
\begin{enumerate}
\item[(i)] If the induced action of $G$ on singular cohomology preserves the class of the Calabi--Yau form
$[\omega_X] \in H^0(X, \Omega_X^n)$, then $D^b(X)_G$ is Calabi--Yau of dimension~$n$.
\item[(ii)] Suppose that, moreover, we have an equivalence $D^b(X)_{G} \cong D^b(X')$ for a variety $X'$.
The induced action of $G^{\vee}$ on $H^{\ast}(X', \BC)$
preserves the class of $\omega_{X'}$.
\end{enumerate}
\end{prop}

\subsection{Equivariant Fourier--Mukai transforms} \label{subsec:equi FM transform}
Let $X$ and $Y$ be smooth projective varieties
and let $G$ be a finite group which acts on $D^b(X)$.
By Lemma~\ref{lem:FM-action} this action is given by Fourier--Mukai transforms
and hence defines an action by Fourier--Mukai transforms on
$D^b(X \times Y)$, see Section~\ref{subsec:eq-pushforward and pullback}.\footnote{Take $\beta$ to be $Y \to \Spec(\BC)$.}
Since this action is pulled back from $X$, 
we often write $G \times 1$ for the group which acts on $D^b(X \times Y)$.

Consider the projections
$X \xleftarrow{\rho} X \times Y \xrightarrow{\pi} Y$.
The (equivariant) Fourier--Mukai transform 
$\F_{\CE} \colon D^b(Y) \to D^b(X)_{G}$
with kernel $\CE \in  D^b(X \times Y)_{G \times 1}$
is defined by
\[ \F_{\CE} A = \rho_{\ast}(\pi^{\ast}(A) \otimes \CE) \]
where the tensor product takes values in $D^b(X \times Y)_{G \times 1}$
and $\rho_{\ast}$ is the equivariant pushforward.
Similarly, the (reverse) equivariant Fourier--Mukai transform
$\G_{\CE} \colon D^b(X)_G \to D^b(Y)$
is defined by
\[ \G_{\CE}(E,\phi) = \hom_{\pi}\left( \CE, \rho^{\ast}(E,\phi) \right)^G \]
where we used equivariant pullback and the $\pi$-relative Hom of Section~\ref{subsubsection_tensor_product}.

\begin{lemma} \label{lemma_adjoints} For any $\CE \in D^b(X \times Y)_{G \times 1}$ let
\[ 
\CE_L = \CE \otimes \rho^{\ast} \omega_{X}^{\vee}[- \dim X],
\quad \ 
\CE_R = \CE \otimes \pi^{\ast} \omega_{Y}^{\vee}[- \dim Y].
\]
Then $\G_{\CE_L}$ and $\G_{\CE_R}$ is the left and right adjoint of $\F_{\CE}$ respectively.
\end{lemma}

Here we followed Remark~\ref{rmk_serrefunc} and have written
$\CE \otimes \rho^{\ast} \omega_{X}^{\vee}[-\dim X]$
for the application of the inverse of the $\pi$-relative Serre functor of $D^b(X \times Y)_{G \times 1}$.

\begin{proof}[Proof of Lemma~\ref{lemma_adjoints}]
For any $(A,\phi) \in D^b(X)$ and $B \in D^b(Y)$ we have
\begin{align*}
& \Hom_{D^b(X)_G}((A,\phi), \F_{\CE}B) \\
& \cong \Hom_{D^b(X \times Y)_{G \times 1}}( \rho^{\ast} (A,\phi), \pi^{\ast}(B) \otimes \CE) \\
& \cong \Hom_{D^b(X \times Y)}(\rho^{\ast}A, \pi^{\ast}(B) \otimes \CE)^G \\
& 
\cong \left( \Hom_{D^b(X \times Y)}(\pi^{\ast}(B) \otimes \CE, \rho^{\ast}(A) \otimes \omega_{X \times Y}[\dim X + \dim Y])^{\vee} \right)^{G} \\
& \cong \left( \Hom_{D^b(Y)}( B , \hom_{\pi}( \CE, \rho^{\ast}(A) \otimes \omega_{X \times Y}[\dim X + \dim Y]))^{\vee}\right)^{G} \\
& \cong \Hom_{D^b(Y)}( \hom_{\pi}( \CE, \rho^{\ast}(A) \otimes \rho^\ast\omega_{X}[\dim X]), B)^G \\
& \cong \Hom_{D^b(Y)}( \G_{\CE \otimes \rho^\ast \omega_{X}^{\vee}[-\dim X]}(A), B).
\end{align*}
The other case is similar.
\end{proof}

We have the following criterion when a Fourier--Mukai transform $\F_{\CE} \colon D^b(Y) \to D^b(X)_G$ is an equivalence.

\begin{prop} \label{prop_equivalence_criterion}
Let $\CE \in D^b(X \times Y)_{G \times 1}$. Assume that
\begin{enumerate}
\item[(i)] $\Hom_{D^b(X)_G}( \CE_x, \CE_y[i] ) = \Hom_{D^b(Y)}( \BC_x, \BC_y[i] )$ for all $x,y \in Y$.
\item[(ii)] $D^b(X)_G$ is indecomposable.
\item[(iii)] The functor $\F_{\CE}$ commutes on objects with Serre functors, i.e.\ $\tilde{S} \F_{\CE}(A) \cong F_{\CE} S(A)$ for all $A\in D^b(Y)$.
\end{enumerate}
Then $\F_{\CE}$ is an equivalence.
\end{prop}
\begin{proof}
By Lemma~\ref{lemma_adjoints} the functor $\F_{\CE} \colon D^b(Y) \to D^b(X)_G$ has both right and left adjoints. The assertion then follows from \cite[Thm.\ 2.3]{BKR}.
\end{proof}


\section{Proof of results} \label{sec:symplectic surfaces}
Let $S$ be a symplectic surface with a $G$-action on $D^b(S)$ satisfying conditions (i)-(iii) of Section~\ref{subsec:main results}
and let $\sigma \in \Stab^{\dagger}(S)$ be a $G$-fixed stability condition.

\subsection{Preliminaries}
We have the following structure result.

\begin{prop} \label{prop_indecomposable}
The equivariant category $D^b(S)_G$ is triangulated, indecomposable and Calabi--Yau of dimension $2$.
\end{prop}
\begin{proof}
Write $\sigma = (\CA,Z)$. 
Since the actions of $\widetilde{\textup{GL}^+}(2,\BR)$ and $G$ on the stability manifold commute,
by Proposition~\ref{prop:DA=DS} we may assume that 
$D^b(\CA) \cong D^b(S)$.
Applying Proposition~\ref{prop_dgenhancement} we see that $D^b(S)_G$ is triangulated and that the $G$-action on $D^b(S)$ lifts to an action on the dg-enhancement.
Hence by Proposition~\ref{prop:calabi-yau} and assumption (i) the category $D^b(S)_G$ is Calabi--Yau.
Since $G$ acts faithfully, the indecomposability of $D^b(S)_G$ holds by definition.
\end{proof} 

\subsection{Moduli spaces}
By work of Toda~\cite{Toda1} the distinguished component $\Stab^{\dagger}(S)$
satisfies condition $(\dagger)$ of Section~\ref{sec:summary}.
Hence by Theorem~\ref{thm:existence gms} we have the following.

\begin{prop} \label{prop_gms} 
Let $v' \in K(D^b(S)_G)$.
Then $\CM_{\sigma_G}(v')$ is a universally closed Artin stack of finite type over $\BC$ which
admits a proper good moduli space. 
\end{prop}

Recall the notion of a $(G,\sigma)$-generic class from Definition~\ref{defn_G_sigma_generic}.

\begin{prop} \label{prop_sigmaG} 
If $v \in \Lambda^G$ is $(G,\sigma)$-generic, then $\CM_{\sigma}(v)^G$
has a good moduli space $N$ which is smooth, symplectic and proper.
The map $\pi \colon \CM_{\sigma}(v)^G \to N$ is a $\BG_m$-gerbe.
\end{prop}

\begin{proof}
By arguing as in the proof of Lemma~\ref{lem:StabCondtoAlgebraicinGinvariant} we can deform $\sigma$ inside $\Stab^{\dagger}(S)^G$ 
to an algebraic stability condition, without modifying the moduli functor $\CM_{\sigma}(v)$.
Together with Remark~\ref{rmk:symp_heart_algebraic} we hence can assume that 
$\sigma$ is algebraic and that $D^b(\CA) \cong D^b(S)$.

Let $\pi \colon \CM_{\sigma}(v)^G \to N$ be the good moduli space of $\CM_{\sigma}(v)^G$.
For every $x \in \CM_{\sigma}(v)^G(T)$ over a scheme $T$ corresponding to an equivariant object $(E,\phi)$ we have an inclusion $\BG_m(T) \hookrightarrow \Aut(x)$ by sending $f \in \BG_m(T)$ to $f \cdot \id_E$.
Moreover, for every $\BC$-point $p \in \CM_{\sigma}(v)^G$ by
Lemma~\ref{lemma_key} we have
\begin{equation*} \label{automorphisms}
\Aut_{\CM_\sigma(v)^G}(p) = 
\Aut_{\CM_{\sigma_G}(v')}(p)
= \Aut_{\CA_G}(E, \phi) = \BC^{\ast} \cdot \id.
\end{equation*}
This shows that $\pi$ is a $\BG_m$-gerbe.

Let $p \in \CM_{\sigma}(v)^G$ be a $\BC$-valued point corresponding to some object $(E,\phi) \in \CA_G$. Let $v' \in K(\CA_G)$ be the class of $(E,\phi)$.
Applying Lemma~\ref{lemma_key} again we have
\[ \Hom_{\CA_G}( (E, \phi), (E, \phi) ) = \BC. \]
Since  $D^b(S)_G$ is Calabi--Yau of dimension $2$, we find that
\[ \Ext_{\CA_G}^2( (E,\phi), (E,\phi) ) = \Hom_{\CA_G}( (E,\phi), (E,\phi))^{\vee} \cong \BC. \]
By Lemma~\ref{Lemma_locally_constant_euler_char}
the Euler characteristic $\chi((E,\phi), (E,\phi))$ is locally constant and hence depends only on $v'$. We write $\chi(v',v')$ for its value.
By Proposition~\ref{def_obst_thy} we conclude that the dimension of the tangent space of $N$ at $p$ is
\[ \dim T_{N,p} = \dim \Ext_{\CA_G}^1( (E,\phi), (E,\phi) ) = - \chi(v',v') + 2. \]
In particular, the dimension is locally constant in $p$.
Moreover, from the $G$-invariant inclusion
$\BC \cdot \id \subset \Hom(E,E)$
we obtain via Serre duality a $G$-invariant surjection $\Ext^2(E,E) \to \BC$
which is precisely the trace map.
This shows that the trace map is surjective on the $G$-invariant part
and thus that the trace-free part vanishes: 
\[ \Ext^2(E,E)^G_0 = 0. \]
Using Proposition~\ref{def_obst_thy} again we find that all obstructions vanish and $N$ is smooth.

The symplectic form on $N$ can be constructed from the fact that it is a moduli space of stable objects in a $2$-CY category.
It can be seen also directly:

Recall from \cite[Sec.\ 10]{HuyLehn} the anti-symmetric Yoneda pairing on $\CM_\sigma(v)$,
\begin{equation} \ext^1_{\rho}(\CE, \CE) \times \ext^1_{\rho}(\CE,\CE) \to \ext^2_{\rho}(\CE, \CE), \label{yoneda} \end{equation}
where $\CE$ is the universal family on $S \times \CM_\sigma(v)$ and $\rho \colon S \times \CM_{\sigma}(v) \to \CM_{\sigma}(v)$ is the projection to the second factor.
Restricting to the $G$-invariant part and pulling back \eqref{yoneda} via $\epsilon \colon \CM_{\sigma_G}(v') \to \CM_{\sigma}(v)$ yields a pairing
\begin{equation} \epsilon^{\ast} \ext_{\rho}^1(\CE, \CE)^G \times \epsilon^{\ast} \ext_{\rho}^1(\CE,\CE)^G \to \epsilon^{\ast} \ext^2_{\rho}(\CE, \CE). \label{aaaghbfdb} \end{equation}
By Proposition~\ref{def_obst_thy} the sheaf $\epsilon^{\ast} \ext^1_{\rho}(\CE, \CE)^G$ is the tangent bundle of $N$.
Since the symplectic form is $G$-invariant, the image of \eqref{aaaghbfdb} is the $G$-invariant part
$\epsilon^{\ast}_{\rho} \ext^2(\CE, \CE)^G = \CO_N$.
Equivariant Serre duality implies that the pairing \eqref{aaaghbfdb} is non-degenerate and hence a symplectic form.
\end{proof}

\subsection{Proof of Theorem~\ref{mainthm1}}
Consider the $G^{\vee}$-torsor given in \eqref{thm:summary:map},
\begin{equation} \label{equation_proof_thm1} 
\bigsqcup_{p_{\ast} v' = v}
M_{\sigma_G}(v') \to M^G.
\end{equation}
Let $F \subset M^G$ be a $G$-linearizable $2$-dimensional component  and let
\[ S' \subset M_{\sigma_G}(v') \]
be a connected component
which maps to $F$. 
The map $S' \to F$ is a torsor for the subgroup of $G^{\vee}$ that preserves this component.

By the second part of Proposition~\ref{prop_fixed_stack_for_fine_moduli_space}
the moduli space $M_{\sigma_G}(v')$ is fine,
i.e.\ there is a universal equivariant object
on $M_{\sigma_G}(v') \times S$.
Let 
\[ \CE = (E,\phi) \in D^b( S' \times S)_{1 \times G}. \]
be its restriction to $S' \times S$.
We will check that the induced Fourier--Mukai transform
\[ \F_{\CE} \colon D^b( S' ) \to D^b(S)_{G} \]
is an equivalence.

For any $x \in S'$ we have
\begin{align*}
\Hom_{D^b(S)_G}( \CE_x, \CE_x ) & = \Hom_{D^b(S)}( \CE_x, \CE_x )^G = \BC \\
\Ext^1_{D^b(S)_G}( \CE_x, \CE_x ) & = \Ext_{D^b(S)}^1(\CE_x, \CE_x)^G = T_{S',x} \cong \BC^2 \\
\Ext^2_{D^b(S)_G}( \CE_x, \CE_x ) & = \Hom_{D^b(S)_G}( \CE_x, \CE_x )^{\vee} \cong \BC.
\end{align*}
The first line follows from the stability of $\CE_x$.
The second line follows from Proposition~\ref{def_obst_thy},
the smoothness of $S'$, and since $F$ and hence $S'$ are $2$-dimensional.
The third line follows since the equivariant category is Calabi--Yau.
In particular, we have 
$\chi(\CE_x, \CE_x) = 0$,
and using Lemma~\ref{Lemma_locally_constant_euler_char} this yields
\[ \chi(\CE_x, \CE_y) = 0 \quad \text{ for all } x,y \in S'. \]
Further for all distinct $x,y \in S'$ by the stability of $\CE_x$ and $\CE_y$ we have
\begin{gather*}
     \Hom_{D^b(S)_G}(\CE_x, \CE_y) = 0 \\
     \Ext_{D^b(S)_G}^2(\CE_x, \CE_y) = \Hom_{D^b(S)_G}(\CE_y, \CE_x)^{\vee} = 0.
\end{gather*}
Hence from the Euler characteristic calculation we also get $\Ext^1(\CE_x, \CE_y) = 0$.
We have therefore proven that for all $x,y \in S'$ we have
\[ \Hom_{D^b(S')}(\BC_x, \BC_y[i] ) = \Hom_{D^b(S)_G}( \CE_x, \CE_y[i] ). \]

By Proposition~\ref{prop_indecomposable} the category $D^b(S)_G$ is indecomposable and Calabi--Yau of dimension~$2$.
Applying Proposition~\ref{prop_equivalence_criterion} we conclude that $\F_{\CE}$ is an equivalence.
\qed

\subsection{A stronger version of Theorem~\ref{mainthm1}} \label{subsec:stronger version of thm1}
We state a version of Theorem~\ref{mainthm1} where we drop the condition on the moduli space to parametrize only stable objects.
This is useful since not every group action on $D^b(S)$ induces an action on such a moduli space.



\begin{thm} \label{mainthm2}
Let $v \in \Lambda^G_{\textup{alg}}$ be $(G,\sigma)$-generic and let $N$ be the good moduli space of $\CM_{\sigma}(v)^G$.
If $N$ has a $2$-dimensional connected component $S'$, then we have an equivalence
\[ D^b(S',\alpha) \xrightarrow{\cong} D^b(S)_G \]
where $\alpha \in \mathrm{Br}(S')$ is the Brauer class of the gerbe $\pi \colon \CM_{\sigma}(v)^G \to N$ restricted to $S'$.
\end{thm}

\begin{proof}
%
Since $\pi$ (restricted to $\pi^{-1}(S')$) is a $\BG_m$-gerbe with Brauer class $\alpha$,
the universal equivariant object on $\CM_{\sigma_G}(v)^G \times S$ restricted to $\pi^{-1}(S')\times S$ descends to an $\alpha \times 1$-twisted $1 \times G$-equivariant universal family $\CE$ on $S' \times S$. 
Arguing as in Theorem~\ref{mainthm1} shows that the associated Fourier--Mukai transform
$\F_{\CE} \colon D^b(S',\alpha) \to D^b(S)_{G}$
is an equivalence. \end{proof}


\subsection{Proof of Theorem~\ref{mainthm4}}
For every $v' \in R_v$ consider the natural morphism
\begin{equation} \label{etale map} M_{\sigma_G}(v') \to M^G. \end{equation}
By Theorem~\ref{thm:summary} this is a $H$-torsor over a connected component of $M^G$, where $H$ is the stabilizer of $v'$ under the $G^{\vee}$-action on $\Lambda_{(S',\alpha)}$.
In particular, $H$ acts freely on $M_{\sigma_G}(v')$.

Assume first that the induced stability condition $\sigma_G$ lies in the distinguished component $\Stab^{\dagger}(S')$.
Since $S'$ is a K3 surface, this implies that $M_{\sigma_G}(v')$ is an irreducible holomorphic symplectic variety.
By the second part of Proposition~\ref{prop:calabi-yau}
the group $H$ acts symplectically on $M_{\sigma_G}(v')$ and thus by the holomorphic Lefschetz fixed point formula every non-trivial element must have a fixed point.
This shows that $H=1$ and that \eqref{etale map} is an isomorphism onto its image.
In the general case, the main result of \cite{Markman2} implies that
$\oplus_{i} H^{0,i}(M_{\sigma_G}(v'))$ is generated by (the conjugate of) the class of a symplectic form,
so by the holomorphic Lefschetz fixed point formula we again obtain $H=1$.
In any case, the morphism \eqref{equation_proof_thm1} is a trivial $G^{\vee}$-torsor over its image.
Since $G$ is cyclic, every point of $M^G$ is $G$-linearizable hence \eqref{equation_proof_thm1} is also surjective. 
This shows the claim. \qed

\section{Existence and properties of auto-equivalences} \label{sec:symplectic surfaces2}
Let $S$ be a symplectic surface.
In this section we tie up some loose ends
in order to make the theorems we proved in the last section effective in practice.
After some preliminary notation, we will consider the following topics:
\begin{enumerate}
\item[(i)] Given a $G$-fixed distinguished stability condition $\sigma \in \Stab^{\dagger}(S)$
we will show that the induced stability condition is distinguished, at least if the equivalence
arises from a universal family. This is useful, because for distinguished stability conditions the moduli spaces of objects are well-understood.
\item[(ii)] We will prove that any symplectic action on a moduli space of stable objects on a K3 surface
is induced by an action on the derived category (Proposition~\ref{mainprop}).
\end{enumerate}

\subsection{Mukai lattice} \label{subsec:Mukai lattice}
The even cohomology of the symplectic surface $S$,
\[ \Lambda = H^{2 \ast}(S,\BZ) = H^0(S,\BZ) \oplus H^2(S,\BZ) \oplus H^4(S,\BZ), \]
admits a non-degenerate pairing, called the \textit{Mukai pairing}, defined by
\[ \langle (r_1, D_1, n_1) , (r_2, D_2, n_2) \rangle = -r_1 n_2 - r_2 n_1 + \int_S D_1 \cup D_2. \]
We will also write $\alpha \cdot \beta$ for $\langle \alpha, \beta \rangle$.
For any $E,F \in D^b(S)$ we have
\[ v(E) \cdot v(F) = -\chi(E,F) \]
where $v(E) = \ch(E) \sqrt{\td(S)}$ is the \textit{Mukai vector} of $E$.

\subsection{Stability conditions} \label{Subsection_stab_conditions}
Given a stability condition $\sigma = (\CA, Z) \in \Stab^\dagger(S)$ in the distinguished component
we will identify the stability function
\[ Z \colon \Lambda_{\mathrm{alg}} \to \BC \]
with the corresponding element in $\Lambda_{\mathrm{alg}} \otimes \BC$ under the Mukai pairing.

Let $\CP(S) \subset \Lambda_{\mathrm{alg}} \otimes \BC$ be the open subset of elements whose real and imaginary part span a positive-definite $2$-plan,
let $\CP^+(S) \subset \CP(S)$ be the connected component which 
contains $e^{i \omega}$ for an ample class $\omega$, and let
\[ \CP^+_{0}(S) = \CP^{+}(S) \setminus \bigcup_{\substack{\delta \in \Lambda_{\mathrm{alg}} \\ \delta \cdot \delta = -2}} \delta^{\perp}. \]
Bridgeland \cite{BridgelandK3} proved that
\begin{equation} \pi \colon \Stab^\dagger(S) \to \CP_0^{+}(S),\ \sigma = (\CA, Z) \mapsto Z \label{covering_map} \end{equation}
is a covering map. His results were generalized to the twisted case in \cite{HuyMacStelStabGen}.

\subsection{Induced stability conditions} 
\label{sec:induced stability}
Let $\sigma \in \Stab^{\dagger}(S)$ be a stability condition and let $G$ be a finite group which acts on $D^b(S)$.
We assume the conditions (i), (ii) and (iii) of Section~\ref{subsec:main results} are satisfied.
Suppose we are given an equivalence
\[ \F_{\CE} \colon D^b(S',\alpha) \to D^b(S)_G \]
induced from a universal family $\CE$ as in Theorem~\ref{mainthm1} or Theorem~\ref{mainthm2}. 
\begin{prop} \label{prop:induced_Stab}
We have $\F_{\CE}^{-1}(\sigma_G) \in \Stab^{\dagger}(S')$.
\end{prop}

We begin with a description how the Mukai lattices $\Lambda$ and $\Lambda'$ of the surfaces $S$ and $S'$ interact.
Consider the composition of the forgetful and linearization functors with the equivalence $\F_{\CE}$:
\[
\FM_{p(\CE)} = p \circ \F_{\CE}, \quad \FM_{p(\CE)^{\vee}[2]} = \F_{\CE}^{-1} \circ q,
\]
where we have also written $p$ for the forgetful functor of $D^b(S' \times S)_{1 \times G}$.
Passing to cohomology this yields morphisms
\[ p \colon \Lambda' \to \Lambda, \quad q \colon \Lambda \to \Lambda' \]
which are both left and right adjoints of each other. The composition is $pq = \oplus_g g$. 
Let
\[ L \subset \Lambda' \]
denote the saturation of the sublattice $q(\Lambda)$.

Given a lattice $M$ we write $M(n)$ for the lattice obtained by multiplying the intersection form with the integer $n$.
\begin{lemma} \label{lem:cohomology}
We have the finite-index sublattices 
\[ \Lambda^G \oplus (\Lambda^G)^{\perp} \subset \Lambda, \quad \  L \oplus L^{\perp} \subset \Lambda'. \]
The map $p$ vanishes on $L^{\perp}$ and defines an embedding of lattices $p \colon L(|G|) \hookrightarrow \Lambda^G$. 
The map $q$ vanishes on $(\Lambda^G)^{\perp}$ and defines an embedding of lattices $q \colon \Lambda^G(|G|) \hookrightarrow L$.
\end{lemma}
\begin{proof}
The isomorphism of correspondences
\[ \rho_g \circ p(\CE) = (\id \times \rho_g)(p(\CE)) 
\cong p(\CE), \]
shows that the image of $p \colon \Lambda' \to \Lambda$ lies in the invariant lattice $\Lambda^G$.
By adjunction it follows that $q$ vanishes on $(\Lambda^G)^{\perp}$.
In particular, for all
$v',w' \in L$ we can write $v'=q(v)$ and $w' = q(w)$ where $v, w \in \Lambda^G \otimes \BQ$.
We obtain
\[ \langle v',w' \rangle_{\Lambda'} = \langle qv, qw \rangle_{\Lambda'} = \langle v, pq w \rangle_{\Lambda} = |G| \langle v, w \rangle_{\Lambda}. \]
Since $\Lambda^G$ is non-degenerate, this shows that $L$ is non-degenerate
and we have the finite-index sublattice $L \oplus L^{\perp} \subset \Lambda'$.
It also shows that $q$ defines an embedding $\Lambda^G(|G|) \hookrightarrow L$.
Moreover, with the same notation as above we have
\begin{align*}
\langle pv', pw' \rangle_{\Lambda} = \langle pqv, pqw \rangle_{\Lambda}
= |G| \langle v, pqw \rangle_{\Lambda}
= |G| \langle qv, qw \rangle_{\Lambda'} = |G| \langle v',w' \rangle_{\Lambda'}.
\end{align*}
We find that $p$ defines an embedding $L(|G|) \hookrightarrow \Lambda^G$.
For every $w' \in L^{\perp}$ we have $\langle pw', v \rangle_{\Lambda} = \langle w', qv \rangle_{\Lambda'} = 0$ for all $v \in \Lambda$, 
which shows that $pw'= 0$.
\end{proof}

If $G$ is abelian, then one can show that $L$ is the invariant lattice for the action
of the dual group on $D^b(S')$, that is $L=(\Lambda')^{G^\vee}$. 

\begin{proof}[Proof of Proposition~\ref{prop:induced_Stab}]
To ease the notation we assume that the Brauer class $\alpha$ vanishes and hence that we work with the usual derived category $D^b(S')$. The case with non-trivial Brauer class works parallel.

Let $\tau = \F_{\CE}^{-1}(\sigma_G)$.
By construction the functor $\F_{\CE}$ is induced from a universal family $\CE\in D^b(S' \times S)_{1\times G}$ of $\sigma_G$-stable objects.
Since $\CE_x$ is $\sigma_G$-stable for all $x\in S'$, the skyscraper sheaves $\BC_x$ are $\tau$-stable for all $x \in S'$.

Let us consider the central charge $Z_{\tau}$ of the stability condition $\tau$. By definition,
it is given by the composition
\[
Z_{\tau} \colon \Lambda' \xrightarrow{p}\Lambda^G_{\mathrm{alg}}\subset \Lambda_{\mathrm{alg}}\xrightarrow{Z} \BC.
\]
By Lemma~\ref{lem:cohomology} the central charge $Z_{\tau}$ factors over $L$
and the real and imaginary part of $Z_{\tau}$ span a positive-definite 2-plane, because $\Re(Z)$ and $\Im(Z)$ do so.

We want to apply now the reasoning of the proof of \cite[Prop.\ 10.3]{BridgelandK3}.
As in \cite[Sec.\ 10]{BridgelandK3}, there is a unique $g\in \widetilde{\mathrm{GL}^+}(2,\BR)$ such that the central charge of $g\tau$ is of the form $\exp(\beta + i \omega)$ for some $\beta, \omega \in \textup{NS}(S')$ with $\omega^2>0$, and such that the sheaves $\BC_x$ have phase 1. 
Then as in the first step in \cite[Prop.\ 10.3]{BridgelandK3} we apply \cite[Lem.\ 10.1]{BridgelandK3} to conclude that for any curve $C\subset S'$ and torsion sheaf $\CE$ supported on $C$ satisfies $\Im Z_\tau (\CE)>0$ which implies $\omega \cdot [C]>0$. Combining this with $\omega^2>0$ we find that the class $\omega$ is ample. 

Invoking again \cite[Lem.\ 10.1]{BridgelandK3} we find further that the heart $\CB$ of $g\tau$ is the tilt of the torsion pair $(\CT,\CF)$, where $\CT=\Coh(S')\cap \CP(0,1]$ and $\CF=\Coh(S')\cap \CP(-1,0]$ and $\CP$ is the slicing corresponding to $g\tau$ (for more on tilting we refer to Appendix~\ref{sec:app_hearts} or \cite{HappelReitenSmalo}). Arguing as in the second step of the proof of \cite[Prop.\ 10.3]{BridgelandK3} we deduce that the torsion pair $(\CT,\CF)$ coincides with the torsion pair $(\CT_{\omega,\beta},\CF_{\omega,\beta})$ associated with the classes $\omega,\beta$ which is constructed in \cite[Sec.\ 6]{BridgelandK3}. With the notation of \emph{loc.\ cit.\,} this yields that $\CB=\CA(\omega,\beta)$ and therefore $g\tau=\sigma_{\omega,\beta}$. In particular, $\tau \in \Stab^\dagger(S')$ and the proof is finished. 
\end{proof}

\subsection{Proof of Proposition~\ref{mainprop}} \label{subsec:proof mainprop}
Let $S$ be a K3 surface with a stability condition $\sigma' = (\CA',Z') \in \Stab^{\dagger}(S)$.
Let $M$ be a fine\footnote{The case of a coarse moduli space works similarly.} moduli space of $\sigma'$-stable objects of Mukai vector $v \in \Lambda$
and let $G$ be a finite group which acts symplectically on $M$.
Consider the Hodge isometry
\[ \Lambda \supset v^{\perp} \cong H^2(M,\BZ). \]
By \cite[Thm.\ 26]{Mongardi2} 
the induced action of $G$  on $H^2(M,\BZ)$ acts trivially on the discriminant lattice. Hence, the action lifts to an action on $\Lambda$ which fixes the vector $v$ and acts by Hodge isometries.
Since $G$ acts symplectically on $M$, the action on $\Lambda$ preserves the class of the symplectic form.

Let $H \in H^2(M,\BZ)$ be a $G$-invariant ample class (obtained for example by averaging any ample class over its images under $G$). Recall the wall and chamber decomposition of $\Stab^\dagger(S)$ associated to $v$ \cite[Sec.\ 9]{BridgelandK3} and denote by $\CC$ the chamber which contains $\sigma'$. From \cite[Thm.\ 1.2]{BMMMP} we infer that there exists a stability condition $\sigma=(\CA,Z) \in \CC$ such that the associated divisor class $\ell_{\sigma}$ equals the class $H$ (for the construction and properties of the divisor classes $\ell_\sigma$ we refer to \cite{BM}). By definition the central charge $Z$ is contained in the $\BC$-vector space $\Span_\BC \langle H, v \rangle \subset \Lambda \otimes \BC$ and hence fixed by $G$. Moreover, since $\sigma$ and $\sigma'$ lie in the same chamber, the moduli functors $\CM_{\sigma}(v)$ and $\CM_{\sigma'}(v)$ agree. This proves $M=M_{\sigma}(v)$. 

Hence we have obtained a subgroup $G \subset O(\Lambda)$ which acts by Hodge isometries, preserves the class of the symplectic form and $Z$.
An application of \cite[Prop.\ 1.4]{HuybrechtsConway} shows that this action on $\Lambda$ is induced by a subgroup
$G \subset \Aut D^b(S)$
which preserves $\sigma$ and acts symplectically.
Using part (b) of Lemma~\ref{lem:obstruction to action} there is a surjection $\tilde{G} \to G$ from a finite group $\tilde{G}$ which acts on $D^b(S)$
with image $G$ in $\Aut D^b(S)$.
By construction the action of $\tilde{G}$ preserves $\sigma$ and $v$ and hence induces an action on $M = M_{\sigma}(v)$.
Since the restriction map $\Aut(M) \to \mathrm{O}( H^2(M,\BZ) )$ is injective \cite[Lem.\ 7.1.3]{MonThesis}, the action of $\tilde{G}$ on $M$ factors through the given action by $G$.
This proves the first part.

For the second part, assume that $G \subset \Aut M$ is cyclic.
Then the action of $\BZ_n$ on $M$ has at least one fixed point
which corresponds to a $\BZ_n$-invariant simple object $F$.
Hence the claim follows from \cite[Sec.\ 4.8]{Notes}. \qed

\section{Examples} \label{sec:examples}
We consider a series of examples to illustrate our methods.
For simplicity we restrict ourselves mostly to cyclic groups acting on the derived category of a K3 surface.

\subsection{Classification} \label{subsec:classification}
Given a variety $X$ and an element $g \in \Aut H^{\ast}(X,\BC)$ of finite order $n$
we define the \emph{frameshape} of $g$ as the formal symbol
\[ \pi_g = \prod_{a|n} a^{m(a)} \label{frameshape} \]
that encodes the characteristic polynomial of $g$ via 
\[ \det( t \cdot \id- g) = \prod_{a|n} (t^a - 1)^{m(a)}. \]
Symplectic auto-equivalences of K3 surfaces of finite order preserving a stability condition
are neatly classified in terms of the frameshape of their action on cohomology.
There are $42$ frameshapes and at most $82$ $\mathrm{O}_+(\Lambda)$-conjugacy classes
which can occur \cite{CHVZ}. The invariant lattices can be found in \cite[App.\ C]{PV}.
For example, in order $2$ there are three cases:
$1^{8} 2^{8}$, $1^{-8} 2^{16}$, and $2^{12}$,
each in a unique conjugacy class. Symplectic involutions of K3 surfaces have frameshape  $1^8 2^8$,
while the others are strictly of derived nature.

\subsection{The dual action of a geometric involution} \label{subsec:dual action}
Let $\iota \colon S \to S$ be a symplectic involution of a symplectic surface with at least one fixed point
and let $G = \BZ_2$ be the group generated by $\iota$.
Hence we are in one of the following two cases:
\begin{enumerate}
\item[(i)] $S$ is an abelian surface and $\iota$ is multiplication by $(-1)$, or
\item[(ii)] $S$ is a K3 surface and $\iota$ is a \emph{Nikulin involution} \cite{SvG}.
\end{enumerate}
The number $r$ of fixed points of $G$ is $16$ and $8$ respectively,
and in both cases the minimal resolution $S'$ of $S/\BZ_2$ is a K3 surface.
In the fiber diagram
\[
\begin{tikzcd}
Z \ar{r}{\alpha} \ar{d}{\beta} & S' \ar{d} \\
S \ar{r} & S/\BZ_2
\end{tikzcd}
\]
the map $\beta$ is the blowup at the fixed points and $\alpha$ identifies $S'$ with the fixed locus $\Hilb^2(S)^{G}$.
By \cite{BKR} (or Theorem~\ref{mainthm1}) we have the equivalence
$\Phi = \beta_{\ast} \alpha^{\ast} \colon D^b(S') \to D^b(S)_G$.

Let $\mathsf{Q} \colon D^b(S') \to D^b(S')$ be the involution given by the action of
the dual group $G^{\vee}$. By applying both sides to skyscraper sheaves one finds\footnote{See also \cite{KM} for a related discussion of this involution.}
\[
\mathsf{Q} = \T_{\CO_S(-\delta)} \circ \prod_{i=1}^{r} \ST_{\CO_{E_i}(-2)}
\]
where
we let
$\ST_{E}(F) = \mathrm{Cone}( \Hom^{\bullet}(E,F) \otimes E \to F )$
denote the spherical twist by the spherical object $E$,
and $\T_{\CL}(E) = E \otimes \CL$
is the twist by a line bundle $\CL$.
The $E_i$ are the exceptional divisors of the resolution $S'$ and
$\delta = \frac{1}{2} \sum_{i=1}^{r} E_i$.

The involution $\mathsf{Q}$ fixes skyscraper sheaves of points not on the exceptional divisor
and sends $\CO_{S'}$ to $\CO_{S'}(\delta)$
as well as $\CO_{E_i}(-1)$ to $\CO_{E_i}(-2)[1]$. For $x \in E_i$ the action exchanges the two distinguished triangles
\begin{equation} \label{abc}
\begin{gathered} 
 \CO_{E_i}(-1) \to \BC_x \to \CO_{E_i}(-2)[1]  \\
\CO_{E_i}(-2)[1] \to \mathsf{Q}(\BC_x) \to \CO_{E_i}(-1).
\end{gathered}
\end{equation}
The frameshape of $\mathsf{Q}$ is
$1^{-8} 2^{16}$ if $S$ is an abelian surface, and $1^8 2^8$ if $S$ is a K3 surface.\footnote{
On the Mukai lattice the involution $\mathsf{Q}$ acts by
\[ (1,0,0) \mapsto (1,\delta, -r/4), \quad (0,E_i, 0) \mapsto (0,-E_i, 1), \quad (0,0,1) \mapsto (0,0,1). \]
}

As an example of a fixed stack computation, consider the moduli space
\[ \CM = \CM_{\sigma_G}(0,0,1) \]
where $\sigma_G$ is induced by a $G$-fixed stability condition on $D^b(S)$
which is equivalent to Gieseker stability for the Mukai vector $v=(0,0,1)$.
The $\BC$-points of $\CM$ correspond to the objects
\[ \BC_x \text{ for all } x \in S', \quad \mathsf{Q}(\BC_x) \text{ for all } x \in E_i, \quad \CO_{E_i}(-1) \oplus \CO_{E_i}(-2)[1]. \]
In this list the $\BC_x$ for all $x \notin E_i$ and the $\CO_{E_i}(-1) \oplus \CO_{E_i}(-2)[1]$
are invariant under $\mathsf{Q}$.
Every $\BC_x$ for $x \notin E_i$ admits two distinct $G^{\vee}$-linearizations, while
$\CO_{E_i}(-1) \oplus \CO_{E_i}(-2)[1]$ admits only one.
We find that the good moduli space of $\CM$ is the quotient $S/\BZ_2$, and that the good moduli space of the fixed stack $\CM^{G^{\vee}}$ is $S$.
Moreover, the forgetful map $\epsilon \colon \CM^{G^{\vee}} \to \CM$ induces the quotient map $S \to S/\BZ_2$ on good moduli spaces.
Applying Theorem~\ref{mainthm2} we obtain the equivalence
\begin{equation} D^b(S) \xrightarrow{\cong} D^b(S')_{G^{\vee}} \label{SSprime equi} \end{equation}
where the Brauer class $\alpha$ is trivial since $S/\BZ_2$ is a fine moduli space away from the singularities. (The equivalence \eqref{SSprime equi} also follows by
a result of Elagin 
\cite[Thm.\ 1.3]{Elagin}.)

Among other things this example shows that while the good moduli space of $\CM$ may be singular, its fixed stack has a smooth proper good moduli space (as guaranteed by Proposition~\ref{prop_sigmaG}).
We also see that $\epsilon$ is not proper, because it does not satisfy the valuative criterion of properness.

\subsection{Involutions on a genus 2 K3 surface} \label{subsec:involution on genus 2}
Let $\pi \colon S \to \p^2$ be a K3 surface obtained as the double cover of a sextic plane curve,
and let $g \colon S \to S$ be a symplectic involution which fixes the hyperplane class $H \in \Pic(S)$.
In this section we will determine the fixed locus of the
moduli spaces of Gieseker semistable sheaves with Mukai vector $(0,H,0)$ and $(0,2H,0)$.
As an application we describe the fixed locus of the induced symplectic birational involution of the resolution of $M_{\sigma}(0,2H,0)$ of O'Grady 10 type.

Recall that the involution $g$ descends to an involution $g_{\p^2}$ of $\p^2$ which can be choosen to act by $(x,y,z) \mapsto (-x,y,z)$, see \cite[Sec.\ 3.2]{SvG}.
The fixed locus of $g_{\p^2}$ is $p=(1,0,0)$ and the line $x=0$.
Let $C_0$ be the preimage under $\pi$ of the line $x=0$
and let $C_1$ be the preimage of a generic line of the form $\lambda y+\mu z$.
Let also $C \in |\CO(2H)|$ be a curve that is preserved under $g$ but disjoint from the fixed points $p_i$.
These curves are preserved by $g$ and contain $6$, $2$ and $0$ fixed points respectively.
Consider the quotients
\[ C_0' = C_0 / \BZ_2, \quad C_1' = C_1 / \BZ_2 \quad\text{and}\quad C' = C / \BZ_2 \]
which are rational, elliptic, and of genus $3$ respectively.
After reordering the exceptional divisors one has in $\Pic(S')$ the relations\footnote{We denote the class in the Picard group with the same symbol as the underlying curve.}
\begin{gather*}
C_0' = \frac{1}{2} C' - \frac{1}{2} ( E_3 + \ldots + E_8 ), \quad 
C_1' = \frac{1}{2} C' - \frac{1}{2} ( E_1 + E_2 ).
\end{gather*}
Suppose that $S$ is of minimal Picard rank $9$. Then by \cite[Lem.\ 1.10]{SvG} the Picard group of $S'$ has the $\BZ$-basis $C_1', \delta, E_2, \ldots, E_8$.
The map on cohomology 
$H^{\ast}(S', \BZ) \to H^{\ast}(S, \BZ)$
induced by the composition
$D^b(S') \xrightarrow{\Phi} D^b(S)_G \to D^b(S)$ is given by
\[
1 \mapsto 1- \pt, \quad \pt \mapsto 2 \pt, \quad E_i \mapsto \pt, \quad \delta \mapsto 4 \pt, \quad C' \mapsto 2H, \quad C_1' \mapsto H- \pt
\]
where we let $\pt$ denote the class of a point on both $S$ and $S'$.


Let $\sigma$ be a generic $G$-fixed stability condition on $S$ which for vectors
$(0,kH,0)$ is equivalent to Gieseker stability. 
We consider the moduli spaces $M_{\sigma}(0,kH,0)$ for $k=1,2$ and their fixed loci:
%
Since $H$ is irreducible on $S$, the coarse moduli space $M_{\sigma}(0,H,0)$ is smooth. Hence by Theorem~\ref{mainthm4} (and using the notation given there)
we have
\[ M_{\sigma}(0,H,0)^G = \bigsqcup_{v' \in \overline{R}_H} M_{\sigma_G}(v'). \]
A direct calculation shows that
there is a unique vector in $\overline{R}_H$ of square $0$ given by $C_1'+E_1$, and $28$ vectors of square $-2$. Therefore,
\[ M_{\sigma}(0,H,0)^G = \widetilde{S} \sqcup (28 \text{ points}) \]
where $\widetilde{S} = M_{\sigma_G}(0,C_1'+E_1,0)$ is a smooth K3 surface. This matches the results of \cite{KMO}.


We turn to $M_{\sigma}(0,2H,0)$. 
Since the moduli space contains strictly semistable objects, we can not apply
Theorem~\ref{thm:summary} directly, but have to account for the semistable locus.
We begin by describing the set $R_{2H}$.
It is given by vectors of the form
\[ v' = C' + \sum_{i=1}^{8} a_i E_i + c \pt \]
where all the $a_i$ are either integers or half-integers,
$\sum_i a_i$ is even and $c=-\sum_i a_i/2$.
Moreover, only vectors satisfying
\begin{itemize}
\item $(v')^2 \geq -2$ (equivalently $\sum_i a_i^2 \leq 3$), or
\item $v' = v_1 + v_2$ with $v_i \in R_H$
\end{itemize} 
contribute to $R_{2H}$.
One finds that $\overline{R}_{2H}$ (i.e.\ modulo $\mathsf{Q}$) consists of the following:
\begin{enumerate}
\item[(i)] The vector $C'$ of square $4$. It can be decomposed in
$28$ different ways as a sum $v_1 + v_2$ with $v_1,v_2 \in R_H$ both of square $-2$,
and in a unique way as $v_1 + v_2$ with $v_1,v_2 \in R_H$ both of square $0$ (given as $C_1' + E_i$).
The moduli space $M_{\sigma_G}(C')$ is of dimension $6$. Its singular locus
is the disjoint union of the product variety $\widetilde{S} \times \widetilde{S}$ and $28$ isolated points.
\item[(ii)] $63$ vectors of square $0$. Each vector can be written in $6$ different ways as a sum of two $(-2)$-vectors in $R_H$.
The moduli space in each case is a K3 surface with $6$ singularities of type $A_1$.
\item[(iii)] $56$ vectors of square $0$, each written uniquely as $v_1 + v_2$ where $v_1$ is of square $0$ (equal to $C_1'+E_1$) and $v_2$ is of square $-2$.
In each case we have $M_{\sigma_G}(v') = M_{\sigma_G}(v_1) = \widetilde{S}$. 
\item[(iv)] $1$ vector of square $0$ obtained as $2 v_1$, where $v_1 = C_1'+E_1 \in R_H$ is of square $0$.
The good moduli space $M_{\sigma_G}(2v_1)$ is $\Sym^2 M_{\sigma_G}(v_1) = \Sym^2 \tilde{S}$.
\item[(v)] $378$ vectors of square $-4$ written uniquely as $v_1 + v_2$ where $v_1, v_2 \in R_H$ are both of square $-2$.
The good moduli space is a point.
\item[(vi)] $28$ vectors of square $-8$ obtained as $2v$, where $v \in R_H$ is of square $-2$.
The good moduli space is a point.
\end{enumerate}

By considering the possible types of semistable points in $M_{\sigma}(0,2H,0)$ and using that $G$ is cyclic one finds that the image of $\bigsqcup_{v' \in \overline{R}_{2H}} \CM_{\sigma_G}(v')$ in $M_{\sigma}(0,2H,0)$
is precisely the fixed locus we are interested in. 
A basic sublocus of the fixed locus is
\[ \Sym^2 \left( M_{\sigma}(0,H,0)^G  \right) \subset M_{\sigma}(0,2H,0)^G. \]
The scheme $\Sym^2 \left( M_{\sigma}(0,H,0)^G \right)$ consists of
\begin{itemize}
\item[(a)] $1$ copy of $\Sym^2(\tilde{S})$,
\item[(b)] $28$ copies of $\tilde{S}$ corresponding to sheaves $E \oplus F$ with $E \in \tilde{S}$ and $F$ corresponding to one of the $28$ fixed points and
\item[(c)] $\Sym^2(28 \text{ points})$ consisting of $378+28$ points corresponding to the direct sum of distinct and identical stable sheaves respectively.
\end{itemize}
Given distinct $G$-invariant stable sheaves $E,F$ of the same slope, the direct sum $E \oplus F$ admits precisely $|G^{\vee}|^2$ many $G$-linearizations.
Moreover, if distinct $E,F \in M_{\sigma}(0,H,0)$ are isolated points of the fixed locus, then
no equivariant lift of $E \oplus F$ has class $C'$ (since otherwise $(E,\phi) = \mathsf{Q} (F,\phi)$, so $E=F$).
We see that the $378$ points in (c) are the image of the points (v), but also of the $6 \cdot 63$ singular points on the K3 surfaces in (ii).

Similarly, the $28$ K3 surfaces in (b) are the image of the $56$ K3 surfaces in (iii).
Since there are precisely $4$ linearizations, these K3 surfaces can not appear in the image of other components, and so yield connected components of $M_{\sigma}(0,2H,0)^G$.
A direct sum $E \oplus E$ of a stable object $E$ admits precisely $|\Sym^2(G^{\vee})| = \binom{|G^{\vee}|+1}{2}$ many linearizations (here $3$).
Hence the $28$ remaining points in (c) are the image of the $28$ points in (vi) and the $28$ isolated singularities in (i).
Moreover, if $v_1 \in R_H$ of square $0$, then $M_{\sigma_G}(2v_1) = \Sym^2 M_{\sigma_G}(v_1)$
maps to the same locus as the inclusion
\begin{equation} M_{\sigma_G}(v_1) \times M_{\sigma_G}(\mathsf{Q} v_1) \subset M_{\sigma_G}(0,C',0). \label{dfsdf22} \end{equation}
Hence the image of $M_{\sigma_G}(2v_1)$ lies in the image of the main component $M_{\sigma_G}(0,C',0)$.
The $63$ moduli spaces in (ii) contain stable points and since we have already taken the coset modulo $\mathsf{Q}$, they must embed
into $M_{\sigma}(0,2H,0)^G$ as isolated components.
We conclude that
\[ M_{\sigma}(0,2H,0)^G = Y \sqcup (28 \text{ smooth K3s}) \sqcup (63 \text{ K3s with 6 nodes}) \]
where $Y$ is the image of $\CM_{\sigma_G}(0,C',0)$ and hence $6$-dimensional.

Rcall that the singular moduli space $M(0,2H,0)$ admits an irreducible holomorphic symplectic resolution $X \to M_{\sigma}(0,2H,0)$ of O'Grady 10 type \cite{OGrady, AS}.
Recall from \cite{SvG} that $\Pic(S) = \BZ H \oplus E_8(-2)$.
Hence there exists $240$ vectors $\alpha \in E_8(-2)$ of square $-4$. 
The involution $g$ acts on these vectors by $g \alpha = -\alpha$.
Let $A \subset E_8(-2)$ be a list of representatives of the orbits of the $(-4)$-vectors under this action.
The singular locus of $M_{\sigma}(0,2H,0)$ is the locus of polystable sheaves, and therefore given by
\[ M_{\sigma}(0,2H,0)^{\mathrm{sing}} =  \Sym^2 M_{\sigma}(0,H,0) \sqcup \bigsqcup_{\alpha \in A} \left( M_{\sigma}(H + \alpha) \times M_{\sigma}(H - \alpha) \right). \]
The resolution $X$ is obtained by a blowup of $M_{\sigma}(0,2H,0)$ along $\Sym^2 M_{\sigma}(0,H,0)$, followed by a resolution of the $120$ isolated points.
The fiber of $X$ over each of these $120$ points is a $\p^5$.
The automorphism
$g \colon  M_{\sigma}(0,2H,0) \to  M_{\sigma}(0,2H,0)$
natural lifts to the blowup (by universal property), but it is not clear a priori whether it lifts along the resolution of the $120$ points.
Hence we only obtain a birational involution
$g' \colon X \dashrightarrow X$
defined away from $120$ disjoint copies of $\p^5$.
We will show the following:
\begin{prop} \label{prop:OGrady10}
The closure of the fixed locus of the birational symplectic involution $g \colon X \dashrightarrow X$ 
is smooth and the disjoint union of one connected component of dimension $6$ containing $120$ copies of $\p^5$, and $119$ K3 surfaces of which $88$ are derived equivalent to $S'$.
\end{prop}

\begin{proof}
The claim follows from our discussion above above and a local analysis of $g$ along $M_{\sigma}(0,2H,0)^{\mathrm{sing}} \cap M_{\sigma}(0,2H,0)^G$
using the local description of the moduli spaces given in \cite[Sec.\ 2]{KLS} and \cite[Sec.\ 3]{AS}.
This is straightforward and we just highlight the main points:
\begin{itemize}
\item The $120$ isolated singular points of $M_{\sigma}(0,2H,0)$
lie in $Y$. They are the images of the stable points of $M_{\sigma_G}(C')$ corresponding to $q(E_{\alpha})$
where $E_{\alpha}$ is the unique stable object in class $H+\alpha$.
The map $g'$ does not extend to the resolution and the closure of the fixed locus of $g'$ contains the whole exceptional $\p^5$.
\item The $63$ K3 surfaces with 6 nodes described in (ii) meet the singular locus of $M_{\sigma}(0,2H,0)$ at the singularities. The corresponding component in the fixed locus of $g'$ is the proper transform and smooth.
\item The $28$ smooth K3 surfaces in $M_{\sigma}(0,2H,0)^G$ corresponding to (iii) lie completely in the singular locus $M_{\sigma}(0,2H,0)^{\mathrm{sing}}$.
The corresponding component in the fixed locus of $g'$ is a trivial $2:1$ cover of this locus and hence given by $56$ K3 surfaces.
\item The K3 surfaces in (iii) and precisely 32 of the K3 surfaces in (ii) arise
as moduli spaces of semistable objects on $S'$
for a Mukai vector $w$ which satisfies $\langle w, \Lambda' \rangle = \BZ$.
Hence all of them are derived equivalent to $S'$. \qedhere
\end{itemize}
\end{proof}

\subsection{An order 3 equivalence} \label{subsec:order 3 equivalence}
Let $E,F$ be elliptic curves defined by cubic equations $f,g$ respectively and consider the cubic fourfold $X \subset \p^5$ defined by the equation $f(x_0,x_1,x_2) + g(x_3, x_4, x_5) = 0$.
Let $\zeta$ be a non-trivial third root of unity.
As in \cite[Ex.\ 1.7(iv)]{Nam} we define a $G = \BZ_3$-action on $X$ by letting the generator act by
\[ (x_0, \ldots, x_5) \mapsto (x_0, x_1, x_2, \zeta x_3, \zeta x_4 , \zeta x_5). \]
The induced action of $G$ on the Fano variety of lines on $X$ has fixed locus
$F(X)^G = E \times F$.
Since $F(X)$ is a moduli space of stable objects in the Kuznetsov component $\mathfrak{A}$ of $D^b(X)$,
and the Kuznetsov component $\mathfrak{A}$ is equivalent to the derived category of a K3 surface by a result of Ouchi \cite{Ouchi},
Theorem~\ref{mainthm1} shows that $\mathfrak{A}_G \cong D^b(A)$ for some connected \'etale cover $A \to E \times F$ of degree $1$ or $2$. In particular, $A$ is an abelian surface.
Theorem~\ref{thm:summary} then determines the fixed loci of the induced action on any smooth $M_{\sigma}(v)$ (with $v \in K(\mathfrak{A})^{\BZ_3}$).

\subsection{Frameshape $2^{12}$}
\label{subsubsec:frameshape212}
We give an example which shows that the equivariant category can behave rather strange.
Consider a symplectic automorphism $\tau \colon S \to S$ of a K3 surface of order $4$,
and let $S'$ be the resolution of the quotient $S/\langle \tau^2 \rangle$.
Since we have taken the quotient only by $\tau^2$, we have a residual involution
$\bar{\tau} \colon S' \to S'$.
The equivalences $\bar{\tau}^{\ast}$ and the dual action $\mathsf{Q}$ of Section~\ref{subsec:dual action} commute and are symplectic.
One checks that the composition $g = \bar{\tau}^{\ast} \circ \mathsf{Q}$ is an involution of $D^b(S')$ of frameshape $2^{12}$.
Then, as a special case of \cite[Sec.\ 4.9]{Notes} the involution $g$ does \emph{not} define an action of $\BZ_2$ on the category,
but defines instead a faithful(!) action of $\BZ_4$.
Moreover one has the equivalence: 
\[ D^b(S')_{\BZ_4} \cong D^b(S'). \]
In other words, the equivariant category under this action is equivalent to the category we started with.
In particular, there does not exist a stable object which is $G$-invariant
and $G$ does not act on any fine moduli space of $S$.\footnote{This
example first appeared in \cite[Sec.\ 4.2]{CHVZ} as a symmetry of K3 non-linear sigma models.
We expect that the behaviour $D^b(S)_G \cong D^b(S)$ is typical of the case where
we have a 'failure of the level-matching condition', i.e.\ $\lambda > 1$ in \cite[App.\ C]{PV}.}

\subsection{Order $11$ equivalences}
Let $g \colon D^b(S) \to D^b(S)$ be a symplectic auto-equiva\-lence of a K3 surface $S$ of order $11$ fixing a stability condition $\sigma \in \Stab^{\dagger}(S)$.
The associated action on cohomology is one of three possible conjugacy classes, each with invariant lattice of rank $4$ \cite[App.\ C]{PV}.
This implies that the pairs $(S,g)$ are isolated points in their moduli space.
Using the Huybrechts--Mongardi criterion \cite{HuybrechtsConway, Mongardi2} each such $g$ induces an automorphism of a moduli spaces $M$ of stable objects in $D^b(S)$.
If we want to determine the equivariant category $D^b(S)_{\BZ_{11}}$ through Theorem~\ref{mainthm1},
we would need to find a $2$-dimensional component of the fixed locus in some $M$.
This seems difficult in this case without studying the concrete geometry. 
By Appendix~\ref{subsec:DTtheory} we can at least read off
the Euler characteristic of the fixed locus:
If $M$ is of dimension $2n$, then $e(M^g)$ is the coefficient of $q^{n-1}$ of the series
\[ \frac{1}{\eta(q)^2 \eta(q^{11})^2} = \frac{1}{q} + 2 + 5q + 10q^{2} + 20q^{3} + 36q^{4} + 65q^{5} + 110q^{6} + O(q^{7}). \]
We hence should expect $2$-dimensional fixed components only in cases where $\dim M \geq 10$.


\appendix

\section{Hearts on symplectic surfaces}
\label{sec:app_hearts}
Let $S$ be a smooth projective symplectic surface
and recall the notation from Section~\ref{Subsection_stab_conditions}.
The goal of this section is to prove the following result:

\begin{prop}
\label{prop:DA=DS}
Let $\sigma \in \Stab^\dagger(S)$ be a stability condition.
Then there exists an element $g \in \widetilde{\textup{GL}^+}(2,\BR)$ such that $g\sigma=(\CA,Z)$ satisfies 
\[ D^b(\CA) \cong D^b(S). \]
\end{prop}

Let us first recall from \cite{BridgelandK3} how the component $\Stab^{\dagger}(S)$ is built up.
First one considers the set $V(S)$ of stability conditions
$\sigma_{\omega,\beta}=(\CA_{\omega,\beta},Z_{\omega,\beta})$ with central charge $Z_{\omega,\beta}=\langle \exp(\beta + i \omega) , \_ \rangle$
where $\beta, \omega \in \mathrm{NS}(S) \otimes \BR$ with $\omega$ ample.
The heart $\CA_{\omega,\beta}$ is obtained from the torsion pair $(\CT_{\omega, \beta}, \CF_{\omega,\beta})$ of $\Coh(S)$ by tilting, see \cite[Sec.\ 6]{BridgelandK3}.
Next, let $U(S)$ be the orbit of $V(S)$ under the free action of $\widetilde{\textup{GL}^+}(2,\mathbb{R})$ on $\Stab^{\dagger}(S)$.
Elements in $U(S)$ are characterized as those stability conditions in $\Stab^{\dagger}(S)$ such that all skyscraper sheaves are stable of the same phase.
Finally, a detailed analysis of the boundary $\partial U(S)$ \cite[Thm.\ 12.1]{BridgelandK3} yields that
any $\sigma \in \Stab^\dagger(S)$ can be mapped into $\overline{U(S)}$ using (squares of) spherical twists.
If $S$ is an abelian surface, then we even have $U(S)=\Stab^\dagger(S)$ \cite[Thm.\ 15.2]{BridgelandK3}.

We start the proof by considering the set of geometric stability conditions $V(S)$.

\begin{lemma} \label{lemma_heart_VS}
For all $\sigma = (\CA, Z) \in V(S)$ we have $D^b(\CA) \cong D^b(S)$.
\end{lemma}
\begin{proof}
Recall that a torsion pair $(\CT,\CF)$ of an abelian category $\CC$ is called cotilting, if for all $E \in \CC$ there is a surjection $F \twoheadrightarrow E$ with $F \in \CF$.
By \cite[Prop.\ 5.4.3]{BondalVandenBergh}, which is a refined version of \cite{HappelReitenSmalo}, for any cotilting torsion pair $(\CT,\CF)$ one has 
$D^b(\CC') \cong D^b(\CC)$, where $\CC'$ is the tilt along $(\CT, \CF)$.

If $\sigma_{\omega,\beta} \in V(S)$, then its heart $\CA_{\omega,\beta}$ is obtained from $\Coh(S)$ by tilting along
the torsion pair $(\CT_{\omega,\beta}, \CF_{\omega,\beta})$.
Huybrechts proved in \cite[Prop.\ 1.2]{HuybrechtsDerivedAbelian} that this torsion pair is cotilting.
\end{proof}

\begin{prop} \label{prop:DADSfortilt}
Let $\sigma \in V(S)$ and let $\CP$ be the associated slicing.
Then for all $a \in \BR$ there is a natural derived equivalence $D^b(\CP(a,a+1]) \cong D^b(S)$. 
\end{prop}

Since Lemma~\ref{lemma_heart_VS} proves the assertion for $a=0$ and the property is preserved by shifts, we only need to consider the case $a\in (0,1)$. 
Write $\sigma=(\CA_{\omega,\beta},Z_{\omega,\beta})$ and $\CA \coloneqq \CP(a,a+1]$.
Then
$$\CA \subset \langle \CA_{\omega,\beta},\CA_{\omega,\beta}[1]\rangle$$ and $\CA$ is a tilt of $\CA_{\omega,\beta}$ for the torsion pair $\CT=\CA_{\omega,\beta}\cap \CA=\CP(a,1]$ and $\CF=\CA_{\omega,\beta}\cap \CA[-1]=\CP(0,a]$. There is a natural exact functor
\[
\Phi \colon D^b(\CA) \to D^b(\CAob) \cong D^b(S)
\]
of triangulated categories \cite[Sec.\ 7.3]{Noohi}. The proof given below shows that this functor defines a derived equivalence. 

\begin{proof}[Proof of Proposition~\ref{prop:DADSfortilt}]
The main idea in the proof is to show that $\Phi$ is essentially surjective. For this we make first some observations. 

Take a very ample line bundle $\CO(1)$. The line bundle $\CO(-i)$ will lie in $\CF_{\omega,\beta}$ for $i\gg0$. Recall from \cite[Sec.\ 6]{BridgelandK3} that the central charge $\Zob$ of the stability condition $\sigma_{\omega,\beta}$ sends an object $E \in D^b(S)$ with Mukai vector $v(E)=(r,l,s)$ to
\begin{align}
\label{eq:centralcharge}
Z_{\omega,\beta}(E)=-s +\frac{r}{2}(\omega^2-\beta^2) + l\beta +i(l\omega -r\omega \beta).
\end{align}
Thus there exists an $i_0$ such that for all $i \geq i_0$ the object $\CO(-i)[1]$ lies in $\CP(0,a]$.
Let us assume (after relabelling) that already $i_0=1$ is sufficient.  

Consider a morphism of sheaves
\[
\CO(-i)^{\oplus m} \xrightarrow{\alpha} \CO(-j)^{\oplus n}.
\]
Since $\CF_{\omega,\beta}$ is the free part of a torsion pair and hence closed under subobjects,
the kernel $K = \mathrm{Ker}(\alpha)$ lies in $\CF_{\omega,\beta}$.
Similarly, $R=\textup{Image}(\alpha)$ is a subsheaf of $\CO(-j)^{\oplus n}$ and lies in $\CF_{\omega,\beta}$.
Therefore the distinguished triangle
\[
K[1] \to \CO(-i)^{\oplus m}[1] \to R [1]
\]
in $D^b(S)$ yields a short exact sequence in $\CP(0,1]$. In particular, $K[1] \in \CP(0,a]$. 

Let $E \in D^b(S)$ be an object.
Using the line bundles $\CO(-i)$ we can find a quasi-isomorphism $O_E\xrightarrow{\simeq} E$ in the homotopy category $K(S)=K(\Coh(S))$, where $O_E=(\dots O_E^{i-1}\to O_E^{i}\to \dots)$ is a (possibly only bounded above) complex whose components are all direct sums of the line bundles $\CO(-i)$ for $i>0$.
Let $c$ be the smallest integer such that the cohomology $\CH^c(E) \in \Coh(S)$ is not isomorphic to zero. Define a new complex 
\[
F_E= ( \dots \, 0 \to \Ker(\partial^{c-1}) \to O_E^c \to O_E^{c+1}\to \, \dots ).
\]
This is a subcomplex of $O_E$ which is bounded and the composition yields a quasi-isomorphism $F_E \xrightarrow{\simeq} E$. 

From the above discussion we infer that $F_E[1]$ is a bounded complex whose components all lie inside $\CP(0,a]$. In particular, the complex $F_E[2]$ viewed inside $K^b(\CP(1,1+a])$ is an element in $D^b(\CA)$. This shows that the realization functor 
\[
\Phi \colon D^b(\CA) \to D^b(\CP(0,1])\cong D^b(S)
\]
is essentially surjective. Invoking 
\cite[Thm.\ A]{ChenHanZhou}
finishes the proof. 
\end{proof}

\begin{cor} \label{cor45435}
\label{cor:US}
For all $\sigma = (\CA,Z) \in U(S)$ we have $D^b(\CA) \cong D^b(S)$.
\end{cor}
\begin{proof}
Any $\sigma \in U(S)$ is a $\widetilde{\textup{GL}^+}(2,\mathbb{R})$-translate of a unique $\tau \in V(S)$.
Thus we have $\CA=\CP(a,a+1]$ for some $a\in \BR$, where $\CP$ is the slicing corresponding to $\tau$.
The assertion follows from Proposition~\ref{prop:DADSfortilt}.
\end{proof}

\begin{proof}[Proof of Proposition~\ref{prop:DA=DS}]
Corollary~\ref{cor:US} proves the assertion for abelian surfaces. Hence we can assume that $S$ is a K3 surface.

If $\Phi \colon D^b(S) \to D^b(S)$ is a derived auto-equivalence and $\CA \subset D^b(S)$ is a heart, then the restriction
$\Phi|_{\CA} \colon \CA \to \Phi(\CA)$ induces an equivalence $D^b(\CA) \cong D^b(\Phi(\CA))$.
Hence $D^b(\CA) \cong D^b(S)$ if and only of $D^b(\Phi(\CA)) \cong D^b(S)$.
Moreover any auto-equivalence commutes with the $\widetilde{\textup{GL}^+}(2,\BR)$-action.
Since, as discussed earlier, any stability condition in $\Stab^{\dagger}(S)$ can be mapped by an auto-equivalence into the closure of $U(S)$,
and we know the claim for elements in the interior of $U(S)$ by Corollary~\ref{cor45435}, we may therefore assume that $\sigma$ lies on the boundary of $U(S)$. 

As $\sigma$ is contained in $\overline{U(S)}$, all skyscraper sheaves $\BC_x$ are semistable. After applying an element of $\widetilde{\textup{GL}^+}(2,\BR)$ we may further assume
that all skyscraper sheaves have phase $1$ with respect to $\sigma$.

Following ideas of \cite{BayerShort}
we will consider a stability condition $\sigma'=(\CA', Z') \in U(S)$ such that skyscraper sheaves have slope $1$
and approach $\sigma = (\CA,Z) \in \partial U(S)$ by first deforming only the real part of $Z'$ and afterwards the imaginary part of the central charge

Concretely, consider the covering map $\pi \colon \Stab^\dagger(S)\to \CP^+_0(S) \subset \Lambda^G_{\mathrm{alg}}\otimes \BC$
and choose an open ball $B \subset \CP^+_0(S)$ of small radius containing $Z$.
Choose a stability condition $\sigma'=(\CA', Z') \in U(S)$ such that skyscraper sheaves have slope $1$ and such that the line from $Z'$ to $\Re Z + \Im Z'$ and the line from $\Re Z + \Im Z'$ to $Z$ viewed in the vector space $\Lambda^G_{\mathrm{alg}}\otimes \BC$ are contained inside $B$.
Let $\tilde{Z}$ be the stability function $\Re Z + \Im Z'$ and let $\tilde{\sigma}= (\tilde{\CA}, \tilde{Z})$ be the stability condition obtained from the covering property of $\pi$.
By construction all skyscraper sheaves remain of phase $1$ along this deformation from $\sigma$ to $\sigma'$. 

The crucial observation now is that the stability condition $\tilde{\sigma}$ is still contained in the open subset $U(S)$. Indeed, recall that the set $U(S)$ can be characterized as the set of all stability conditions for which all skyscraper sheaves $\BC_x$ are stable of the same phase. 
Assume that a skyscraper sheaf $\BC_x$ becomes unstable along the line segment from $Z'$ to $\tilde{Z}$.
Since semistablity is a closed property, there would have to exist a $\tau$ on this line segment where $\BC_x$ becomes semistable.
Since the imaginary part of the central charges stays constant along the path,
$\BC_x$ is still contained in the abelian category $\CP(1)$, where $\CP$ is the slicing associated to $\tau$.
As $\BC_x$ is semistable, there exists a stable object $F \in \CP(1)$ and a non-zero morphism $F \to \BC_x$ which is not an isomorphism.
Since being stable is an open property \cite[Prop.\ 2.10]{BayerBridgeland}, the object $F$ was also stable for a stability condition on the line segment where $\BC_x$ is stable. However, a morphism between stable objects of the same phase is either an isomorphism or $0$, yielding a contradiction. We conclude that $\tilde{\sigma} \in U(S)$.

Let $\tilde{\CP}$ be the the slicing associated to $\tilde{\sigma}$.
Then as argued in \cite[Lem.\ 5.2]{BayerShort} the abelian category 
$\tilde{\CA}=\tilde{\CP}(1/2, 3/2]$ is constant along
a deformation that only changes the imaginary part of the stability condition.
This yields $\CP(1/2,3/2]=\tilde{\CA}$, where $\CP$ is the slicing associated to $\sigma$.
Let $g \in \widetilde{\textup{GL}^+}(2,\BR)$ denote the rotation by $\pi/2$.
Then $\tilde{\CA}$ is the heart of both $g\tilde{\sigma}$ and $g \sigma$.
Since $\widetilde{\textup{GL}^+}(2,\BR)$ preserves $U(S)$, we have $g\tilde{\sigma} \in U(S)$ and therefore by Corollary~\ref{cor:US} we conclude that $D^b(\tilde{\CA})\cong D^b(S)$.
\end{proof}
\begin{rmk} \label{rmk:symp_heart_algebraic}
Given an algebraic stability condition $\sigma=(\CA,Z) \in \Stab^{\dagger}(S)$, the proof above shows that in Proposition~\ref{prop:DA=DS} one can choose the element $g$ such that $g\sigma$ is algebraic as well.
Indeed, this is immediate for stability conditions which are mapped by some auto-equivalence into $U(S)$.
For $\sigma \in \partial U(S)$, we first applied an element from $\widetilde{\textup{GL}^+}(2,\BR)$ so that skyscraper sheaves get mapped to $-1$ and then applied the rotation by $\pi/2$.
If $\sigma$ is algebraic, both steps can be achieved by multiplying $Z$ with elements from $\BQ+i \BQ$.
\end{rmk}

\section{The Euler characteristic of fixed loci} \label{subsec:DTtheory}
We state a result which may be viewed as a numerical version of Theorems~\ref{mainthm4}:

Let $M = M_{\sigma}(v)$ be a moduli space of stable objects of Mukai vector $v$ on a K3 surface $S$, and let $g \colon M \to M$ be a symplectic automorphism of finite order.
Let $\pi_g = \prod_{a} a^{m(a)}$ be the frameshape of the induced action on the Mukai lattice $\Lambda$ (obtained from lifting the action on $H^2(M,\BZ)$ to $\Lambda$, see Section~\ref{subsec:proof mainprop}).
We define the modular form
\[ f_g(q) = \prod_{a} \eta( q^a )^{m(a)} = q + O(q^2), \]
where $\eta(q) = q^{1/24} \prod_{m \geq 1} (1-q^{m})$ is the Dedekind elliptic function. 
\begin{prop}
$e(M_{\sigma}(v)^g) = \textup{Coefficient of $q^{v \cdot v/2}$ of } f_g(q)^{-1}$.
\end{prop}

Here $e(Z)$ denotes the topological Euler characteristic of a finite type scheme $Z$.
If $M$ is the Hilbert scheme of points and the automorphisms is induced by an automorphisms of the underlying surface,
this follows by a local analysis, see \cite{BO} and also \cite{BG} for the extension to non-cyclic groups.
The general case is evidence for an affirmative answer to Question~\ref{main open question}.

\begin{proof}
We prove the claim by computing the trace of the induced automorphism $g_{\ast} \colon H^{\ast}(M,\BZ) \to H^{\ast}(M,\BZ)$.
The action $g_{\ast}$ is a monodromy operator.
By Mongardi \cite[Thm.\ 26]{Mongardi2} the unique lift of $g^{\ast}|_{H^2(M,\BZ)}$ to an automorphism of the Mukai lattice
fixes $v$. By the classification \cite[App.\ C]{PV} the lift is of determinant $1$.
This shows that $g_{\ast}|_{H^2(M,\BZ)}$ is also of determinant $1$.
Using \cite[Lem.\ 4.13]{Markman} we conclude that $g_{\ast}$ lies in the image of the
canonical representation of the (intergrated) Looijenga--Lunts--Verbitsky (LLV) Lie algebra
\[ \mathrm{SO}(H^2(M,\BC)) \to \mathrm{SO}(H^{\ast}(M,\BC)). \]

It hence remains to prove that, given an element $\varphi \in \mathrm{SO}(H^2(M,\BC))$ of finite order with frameshape $\prod_{a} a^{m(a)}$ (when extended to $\Lambda \otimes \BC$ by the identity on $H^2(M,\BC)^{\perp}$), then the trace of $\varphi$ on $H^{\ast}(M,\BC)$ has the given form.
Since this is a purely topological question, we may assume $M = \Hilb_n(S)$ where $n=(v \cdot v)/2 + 1$.
Moreover, after conjugation by an element in $\mathrm{SO}(H^2(M,\BC))$ 
we may assume that the lifted element $\bar{\varphi} \in \mathrm{SO}(\Lambda \otimes \BC)$ preserves the decomposition by degree and
acts as the identity on $H^0(S,\BC) \oplus H^4(S,\BC)$.
In particular, $\bar{\varphi}$ induces an action on $H^{\ast}(\Hilb_k(S))$ for any $k$.
By the formulas for the LLV Lie algebra action in \cite{Ob} in terms of Nakajima operators,
the Nakajima operators are equivariant with respect to the action of $\bar{\varphi}$ on $\Lambda \otimes \BC$ and $H^{\ast}(\Hilb_k(S))$.
If $V_i$ are the eigenspace of $\bar{\varphi}$ on $\Lambda \otimes \BC$ with eigenvalue $\lambda_i$ this yields
$\oplus_{k \geq 0} H^{\ast}(\Hilb_k(S)) = \otimes_{i=1}^{24} \mathrm{Sym}^{\bullet}(V_i)$ and thus
\[
\sum_{k \geq 0} \mathrm{Tr}\left( \varphi | H^{\ast}(\Hilb_k(S)) \right) q^n = \prod_{m \geq 1} \prod_{i=1}^{24} \frac{1}{(1-\lambda_i q^m)} = \prod_{n \geq 1} \prod_{a \geq 1}  \left( \frac{1}{1-q^{an}} \right)^{m(a)}
\]
where the last equality follows from a direct computation.
\end{proof}

\end{document}